\documentclass[11pt]{article}

\usepackage{verbatim}
\usepackage{hyperref}
\usepackage{enumerate, amsthm}
\usepackage{enumitem}
\usepackage{mathrsfs}
\usepackage{rotating}
\usepackage{cancel}
\usepackage{enumitem}
\usepackage{xcolor}
\usepackage[margin=1 in]{geometry}
\usepackage[margin={1cm, 1cm},font=small, labelsep=endash]{caption}
\usepackage{mathtools}
\usepackage{tikz}
\usetikzlibrary{patterns}
\usepackage{ifthen}

      \usepackage{amssymb}
      \usepackage{amsmath}
      \usepackage{centernot}

      \numberwithin{equation}{section}
      \theoremstyle{plain}
      \newtheorem{theorem}{Theorem}[section]
      \newtheorem{lemma}[theorem]{Lemma}
      \newtheorem{corollary}[theorem]{Corollary}
      \newtheorem{proposition}[theorem]{Proposition}

      \theoremstyle{definition}

      \theoremstyle{remark}
      \newtheorem{remark}[theorem]{Remark}

      \newcommand{\e}{\mathrm{e}}
      \newcommand{\ep}{\varepsilon}
      \renewcommand{\P}{\mathbb P}

      \newcommand{\Z}{\mathbb{Z}}
       
      \newcommand{\N}{\mathbb{N}}
      \newcommand{\n}{\mathbf{n}}
      
      \renewcommand{\L}{\Lambda}
      
      \newcommand{\ising}[1]{\langle #1 \rangle}

      \newcommand{\lr}[4]{#3\xleftrightarrow[#1]{#2} #4}
      \newcommand{\nlr}[4]{#3\mathrel{\mathop{\centernot\longleftrightarrow}_{#1}^{#2}} #4}
      \newcommand{\mb}{\mathbf}

      \makeatletter
      \def\@setcopyright{}
      \def\serieslogo@{}
      \makeatother
      
\begin{document}
      	\title{Exponential decay of truncated correlations for the Ising model in any dimension for all but the critical temperature}
      	
      \author{Hugo Duminil-Copin\footnotemark[1]\footnote{Universit\'e de Gen\`eve} \footnotemark[2]\footnote{Institut des Hautes \'Etudes Scientifiques} , Subhajit Goswami\footnotemark[2] , Aran Raoufi\footnotemark[3]\footnote{ETH Zurich}}	
      	\date{\today}
      	
      	\maketitle
        
                \begin{abstract}
               The truncated two-point function of the ferromagnetic Ising model on $\mathbb Z^d$ ($d\ge3$) in its pure phases is proven to decay exponentially fast throughout the ordered regime ($\beta>\beta_c$ and $h=0$). Together with the previously known results, this implies that the exponential clustering property holds throughout the model's phase diagram except for the critical point: $(\beta,h) =  (\beta_c,0)$.
                \end{abstract}
\section{Introduction} \label{sec:intro}

\subsection{Exponential decay of truncated correlations of the Ising model}
In addition to its original presentation as a model for the phase transition in ferromagnets, the Ising model has attracted attention from a variety of perspectives.  These range from studies of phase transitions exhibited by the equilibrium states to the study of cutoff phenomena and transitions in stochastic processes given for instance by  Glauber dynamics and Metropolis algorithms~\cite{LubSly13}.  Also, universality of critical phenomena in the Ising model justifies the fact that the theory of the Ising model provides information also about many other systems.

As is well known,  sufficiently far from phase transitions,  systems of statistical physics exhibit exponential relaxation of truncated correlations~\cite{DobShl}, in both  the equilibrium and the dynamical sense.  It is more challenging to narrow the range of exceptions to a set of points, or lines,  in the model's phase space.   The main result in this article completes that task for the $d$-dimensional nearest-neighbor ferromagnetic Ising model. The results extend to finite-range Ising models, but we choose to focus on the nearest-neighbor case for simplicity.

To set the notation, let us recall the definition of the model on a graph $G$ with vertex-set $V$ and edge-set $E$.   Associated with the graph's vertex-set is a collection of binary variables 
 $\sigma=(\sigma_x:x\in V)$, with $\sigma_x\in\{-1,1\}$.  
The system's Hamiltonian is given by the function 
\begin{equation} \label{eq:H}
H_ {G,h}(\sigma)~:=~  -\sum_{x\in  V} h \sigma_x \ -\ \sum_{\{x,y\}\subset V } J_{x,y}  \sigma_x\sigma_y\,,
\end{equation}
with  the magnetic field $h$ and coupling constants $J_{x,y}$.   In the case on which we focus here, 
$G$ equals $\Z^d$ (the graph is the regular $d$-dimensional lattice) and 
\begin{equation*} 
J_{x,y} \ = \  \begin{cases}   1 &     \text{ if $\{x,y\}\in E$,  }  \\ 
                   0  &   \text{ otherwise, }
                       \end{cases}  
\end{equation*} 
which corresponds to nearest-neighbor ferromagnetic interactions.

On finite graphs, the Gibbs equilibrium states at inverse temperature  $\beta \in (0, \infty)$  are given by  probability measures on the space of configurations  under which the expected value of a function $f:\{-1,1\}^{V}\rightarrow \mathbb R$ is 
$$\langle f\rangle_{ G,\beta,h}=\frac1{Z( G,\beta,h)}\sum_{\sigma\in\{-1,1\}^{V}}f(\sigma)\exp[-\beta H_ {G,h}(\sigma)]\,  ,$$
where the sum is normalized by the partition function $Z( G,\beta,h)$   so that $\langle 1\rangle_{ G,\beta,h}=1$.  
Examples of Gibbs measures on $\Z^d$ can be constructed as 
weak limits of the finite volume Gibbs measures on finite subgraphs $G\subset \Z^d$ which locally converge to $\Z^d$. We denote the measure thus obtained by $\langle\cdot\rangle_{\beta,h}$.
Also, one  defines
\begin{equation}\label{lim}
\langle \cdot\rangle_{\beta}^+ \  = \   \lim_{h\searrow 0} \langle \cdot\rangle_{\beta,h}  \,. 
\end{equation} 
where the  limit is meant 
 in the ``weak sense'' (i.e.~for the expectation values of local functions of the spins). 
Convergence  can be deduced by   monotonicity arguments based on correlation inequalities, by which one may also establish the existence of $\beta_c=\beta_c(\Z^d)\in[0,\infty]$ such that      
 \begin{eqnarray}  
0 \leq  \beta < \beta_c  &  \Rightarrow & \langle \sigma_x\rangle_{\beta}^+ \ = \  0\ ,\ \forall x\in \Z^d, \notag \\  
 \beta > \beta_c   &  \Rightarrow &   \langle \sigma_x\rangle_{\beta}^+ \ >\  0 \ ,\ \forall x\in \Z^d. 
 \end{eqnarray} 

For a given Gibbs measure $\langle\cdot\rangle$, in finite or infinite volume,  the truncated two-point 
correlation function is defined as:  
$$\langle \sigma_0;\sigma_x\rangle:=\langle \sigma_0\sigma_x\rangle-\langle \sigma_0\rangle\langle\sigma_x\rangle.$$
For $\beta > \beta_c$, there exists a spin-flip symmetric equilibrium state with long-range order, for which the truncated correlations do not decay to zero.   However, the relevant question is the rate of decay of the pure state $\langle \cdot\rangle_{\beta}^+$ and its symmetric image  
$\langle \cdot\rangle_{\beta}^-$.  The main result of this article is the following.

\begin{theorem}\label{thm:main}
For the nearest-neighbor Ising model on $\Z^d$ in  dimension $d \geq 3$,  for any $\beta>\beta_c$ there exists $ c = c(\beta,d) > 0$ such that for every $x,y\in\Z^d$,
\begin{equation}\label{eq:main}
0\le \langle \sigma_x;\sigma_y\rangle_{\beta}^+\le \exp[- c\|x-y\|].
\end{equation} 
\end{theorem}

The previous result holds for any extremal translation invariant Gibbs state, since by \cite{Bod05,Rao17}, those are given by $\langle\cdot\rangle_\beta^+$ and $\langle\cdot\rangle_\beta^-$. As mentioned above, the theorem would extend to finite-range interactions, but we prefer to focus on this context for simplicity of notations.
Jointly with the previously known results, this completes the proof  that for the nearest-neighbor Ising model in any dimension it is   only at the critical point $(h,\beta) = (0,\beta_c)$ that the pure state's truncated two-point function fails to decay exponentially fast.   
The aforementioned statement 
which Theorem~\ref{thm:main}  supplements include:  
\begin{enumerate}
\item   At any $h\neq 0$ the limiting state  is analytic in $h$ and $\beta$, and it exhibits exponential decay of suitably truncated correlations. This was proven by Lebowitz and Penrose in \cite{LP68} using an argument which drew on the model's   Lee-Yang property~\cite{LY52}, or using the random-current representation.
\item 
For $h=0$ and $\beta < \beta_c$  the exponential decay in arbitrary dimension was 
established in \cite{ABF87} (see also \cite{DumTas15} for an alternative proof).   

\item In the converse direction: 
the vanishing of the spontaneous magnetization  at   $(h,\beta) = (0,\beta_c)$ for the nearest neighbor model in any dimension~\cite{AizDumSid15}  together with the lower bound 
\begin{equation}
\sum_{\|x\|_\infty =R}\langle \sigma_0\sigma_x\rangle_{\beta_c}\ \ge \ 1 \, , 
\end{equation} 
which was established by Simon~\cite{Sim80}, imply that for any $d\ge 2$ at the critical point the truncated two-point function does not decay exponentially fast.  
\item And, to mention a last result: the special case of 
$d=2$ is  analyzable through Onsager's exact solution~\cite{McCWu73} or using the Kramers-Wannier duality (the decay of truncated correlations can be obtained via the decay of correlations in the high-temperature dual Ising model). 
\end{enumerate} 

Let us add that the truncated two-point function offers a bound on the decay of more general correlations
which follows easily from the following lemma (whose proof follows from the switching lemma; see the discussion in Section~\ref{subsec:random_current}). For a set of vertices $A$, set $\sigma_A:=\prod_{x\in A}\sigma_x$.

\begin{lemma}\label{lem:00}  For every finite graph $G$, every $\beta>\beta_c$ and every two disjoint sets of vertices $A$ and $B$, we have that
\begin{equation}\label{truncated_relation} 
0 \le  \langle \sigma_A\sigma_B \rangle^+_{G,\beta} - 
\langle \sigma_A \rangle^+_{G,\beta} \langle \sigma_B\rangle ^+_{G,\beta} \le \, 2^{| A| + | B| - 4} \, \sum_{\substack{a\in A\\ b\in B}}  
\langle \sigma_a;\sigma_b \rangle^{+}_{G,\beta}   \, .  
\end{equation} 
 \end{lemma} 

\subsection{Mixing of the Fortuin-Kasteleyn representation of the Ising model}

Theorem~\ref{thm:main} will be derived by studying a related model, called the {\em Fortuin-Kasteleyn (FK) percolation}, or {\rm random-cluster model}. 
FK percolation is one of the most classical generalizations of Bernoulli percolation and electrical networks. This model was introduced by Fortuin and Kasteleyn in \cite{ForKas72} and since then has been the object of intense studies, both physically and mathematically.

A {\em percolation configuration} on a graph $G=(V,E)$ is an element $\omega=(\omega_e:e\in E)$ in $\{0,1\}^{E}$. An edge $e$ is said to be {\em  open} (in $\omega$) if $\omega_e=1$, otherwise it is {\em closed}. A configuration $\omega$ can be seen as a subgraph of $G$ with vertex-set $V$ and edge-set $\{e\in E:\omega_e=1\}$. A {\em cluster} is a connected component of the graph $\omega$. Below, we will write $x\leftrightarrow y$ if $x$ and $y$ are in the same cluster, and $x\leftrightarrow \infty$ if $x$ is in an infinite cluster.

Fix $p\in[0,1]$ and $q\ge1$. 
Let $G$ be a finite subgraph of $\Z^d$ and $\xi$ a configuration on $\Z^d$. Let  $\phi_{G,p,q}^\xi$ be the measure on percolation configurations $\omega$ on $G$ defined by 
$$\phi_{G,p,q}^\xi(\omega)=\frac{1}{Z^\xi(G,p,q)}\big(\tfrac{p}{1-p}\big)^{|\omega|} q^{k_\xi(\omega)},$$
where $|\omega|:=\sum_{e\in E}\omega_e$ and $k_\xi(\omega)$ is the number of clusters 
intersecting $G$ of the percolation configuration $\overline \omega$ on $\Z^d$ defined by 
$\overline \omega_e=\omega_e$ if $e\in E$, and $\xi_e$ if $e\notin E$, and $Z^\xi(G,p,q)$ is 
a normalizing constant making the total mass of the measure equal to 1. We refer to $\xi$ as 
the {\em boundary condition} of $\phi_{G,p,q}^\xi$. In the particular case when $\xi \equiv 
1$ (or $0$) we denote the corresponding measure by $\phi_{G,p,q}^1$ (respectively 
$\phi_{G,p,q}^0$) and call the corresponding boundary condition as {\em wired} (respectively 
{\em free}).

The FK-percolation model with cluster-weight $q=2$ is related to the Ising model via the Edwards-Sokal coupling (see next section) and is therefore referred to in this article as the {\em FK-Ising model}. When $q=2$, it was proved in \cite{Bod06,Rao17} that for every $p\in[0,1]$, there exists a unique infinite-volume measure $\phi_{p,2}$ which is the weak limit of measures 
$\phi_{G,p,2}^\xi$ as $G$ exhausts $\Z^d$. Furthermore, there exists a constant $p_c=p_c(d)$ such that $\phi_{p,2}[0\leftrightarrow\infty]$ is equal to 0 for every $p<p_c$, and is strictly positive for every $p>p_c$. We refer to \cite{Gri06} for a justification that this limit exists. 

Below and in the rest of this paper, we focus on the case $q=2$ and drop it from the notation. We denote $x + [-n, n]^d \cap \Z^d$ by $\L_n(x)$ and the set of edges between two vertices of $\L_n(x)$ by $E_n(x)$. In the particular case when 
$x = 0$, we write $\L_n$ and $E_n$ respectively. The \emph{boundary} of $\L_n(x)$, denoted as $\partial \L_n(x)$, is defined as the set of all vertices in $\L_n(x)$ which have a neighbor in $\Z^d \setminus 
\L_n(x)$. Theorem~\ref{thm:main} is a consequence of the following exponential mixing property.

\begin{theorem}[Exponential mixing]
\label{thm:mixing}
For every $d\ge3$ and $p>p_c$, there exists a constant $c>0$ such that for every $n\ge1$,
\begin{equation}|\phi_{p}[A\cap B] - \phi_{p}[A]\phi_{p}[B]| \leq \exp(-cn),\label{eq:main2}\end{equation}
where $A$ and $B$ are any two events depending on edges in $E_{n}$ and outside $E_{2n}$ respectively. 
\end{theorem}

Before discussing the proof of this theorem, let us  explain how it implies Theorem~\ref{thm:main}.

\begin{proof}[Proof of Theorem~\ref{thm:main}]
Fix $\beta>\beta_c$ and set $p:=1-e^{-2\beta}>p_c$. The Edwards-Sokal coupling (see \eqref{eq:ESinfini} in the next section) gives that for every $x\in\Z^d$,
\begin{align}
\label{eq:main_expansion}
\ising{\sigma_0; \sigma_x}_{\beta}^+ &=\langle\sigma_0\sigma_x\rangle_{\beta}^+-\langle\sigma_0\rangle_\beta^+\langle\sigma_x\rangle_\beta^+= \phi_{p}[0\leftrightarrow x] - \phi_{p}[0\leftrightarrow \infty]\phi_{p}[x\leftrightarrow \infty].\end{align}
Assuming that $x$ is at a graph distance of at least $4n$ of the origin, this implies that
\begin{align}
\langle\sigma_0; \sigma_x\rangle_{\beta}^+ &\le  \phi_{p}[0\leftrightarrow\partial\L_n,x\leftrightarrow\partial\L_n(x)] - \phi_{p}[0\leftrightarrow\partial\L_n]\phi_{p}[x\leftrightarrow\partial\L_n(x)]+2e^{-c'n}\label{eq:pk}\\
&\le \e^{-cn}+2\e^{-c'n},\nonumber\end{align}
where in the second line, we used Theorem~\ref{thm:mixing}, and in the first, the fact that for $p>p_c$, there exists $c'>0$ such that for every $n\ge1$,
\begin{equation}\label{eq:exponential decay FK}\phi_{p}[0\leftrightarrow\partial\L_n,0\nleftrightarrow\infty]\le \e^{-c'n}.\end{equation}
(This fact follows from Theorem (5.104) of \cite{Gri06} when combined with the result of \cite{Bod05}.)
\end{proof}
Let us remark that the exponential decay \eqref{eq:exponential decay FK} for the size of finite clusters in the supercritical FK-Ising model does not directly imply the exponential decay of truncated two-point functions for the Ising model, since the first term on the right of \eqref{eq:pk} involves correlations between the events that $0\leftrightarrow\partial\L_n$ and $x\leftrightarrow\partial\L_n(x)$, and that these correlations could {\em a priori} be 
large.

\smallskip

Let us also remark that our method actually gives a better bound on the error term in Theorem~\ref{thm:mixing} than $\exp(-cn)$. Namely, we obtain that
\begin{equation}
\label{eq:weak_mix}
|\phi_p[A\cap B]-\phi_p[A]\phi_p[B]|\le \exp(-cn)\max_\xi\phi_{\L_n, p}^\xi[A]\phi_p[B]\,.
\end{equation}
This is stronger than the {\em weak mixing} property for FK percolation measures which is obtained by replacing $\max_\xi\phi_{\L_n, p}^\xi[A]$ with 1 but weaker than the {\em ratio 
weak mixing} property where we want to get rid of the maximum over boundary conditions. However our proof of Theorem~\ref{thm:mixing} (see, e.g., \eqref{eq:bk}) also implies that $\phi_p$ has the so-called {\em exponentially bounded controlling regions} in the sense of \cite[p.~455]{Alexander98}. Then the ratio weak mixing property of $\phi_p$ follows from \eqref{eq:weak_mix} and Theorem~3.3. in \cite{Alexander98}. For potential application in future works we present it here as a corollary of Theorem~\ref{thm:mixing}.
\begin{corollary}[Ratio weak mixing]
	\label{cor:mixing}
	For every $d\ge3$ and $p>p_c$, there exists a constant $c>0$ such that for every $n\ge1$,
	\begin{equation}|\phi_{p}[A\cap B]   - \phi_{p}[A]\phi_{p}[B]| \leq \exp(-cn)\phi_{p}[A]\phi_{p}[B],\label{eq:main3}\end{equation}
	where $A$ and $B$ are any two events depending on edges in $E_{n}$ and outside $E_{2n}$ respectively. 
\end{corollary}
\subsection{Idea of the proof}

The core of the proof will be the derivation of the following proposition. 
\begin{proposition}
\label{thm:mixing1}
There exists $c>0$ such that for every integer $N$ that is divisible by 4,
$$\max_{e \in E_{N/4}} \phi_{\L_{N}, p}^1[\omega_e] - \phi_{\L_{N}, p}^0[\omega_e] \leq \exp[-c(\log N)^{1 + c}]\,.$$
\end{proposition}

The important feature of the upper bound above is that it beats any inverse polynomial.
This proposition, together with a coarse-grained argument inspired by Pisztora, implies the exponential mixing. While  the Ising model can be approached through a number of graphical representations (low and high temperature expansions, FK-Ising, random current, etc), which have been used separately in a variety of results, the argument presented here relies in a crucial way on the combination of two such representations:  the random current and the FK-Ising. The random current representation is used to rewrite the difference between $\phi_{\L_{N}, p}^1[\omega_e] $ and $\phi_{\L_{N}, p}^0[\omega_e] $ in terms of the probabilities of non-intersection for currents in a duplicated system of currents. Then, FK-Ising is used to show that this duplicated system of currents is very well-connected, and that the probability of long paths of currents not being connected is quite small.

At different stages of the proof (already in the proof of Theorem~\ref{thm:main} above), essential use is made of  the very helpful result of 
Bodineau~\cite{Bod05} stating that for any $d\ge 3$ the critical parameter $p_c$ coincides with the so-called {\em slab percolation}. This result is combined with the result \cite{Pis96} to implement a coarse-grain argument inspired by Pisztora renormalization. This is used to prove two facts: boxes are connected with excellent probability in the supercritical FK-Ising model, and Theorem~\ref{thm:mixing} follows from Proposition~\ref{thm:mixing1}.

\subsection{Open problems}

Corollary~\ref{cor:mixing} falls short of the {\em ratio-strong mixing property} related to the phenomenon of {\em boundary phase transition} for Ising models (see \cite{MartinelliOlivSchonmann}). Although this stronger property is absent for Ising models in dimensions larger than 2 at low temperature, it is expected to hold in the entire subcritical phase. More precisely, one would like to prove: 
$$|\phi_p[A\cap B] - \phi_{p}[A]\phi_{p}[B]|\le \exp(-c \, d_{A, B})\phi_{p}[A]\phi_{p}[B]$$
where $d_{A, B}$ is the distance between the supports of the events $A$ and $B$.

Another important improvement would be to understand the case of the Potts models with $q\ge3$ colors. While the $\beta<\beta_c$ was recently treated in every dimension in \cite{DumRaoTas17}, the study of the $\beta>\beta_c$ regime is still very limited. In \cite{DumTas17}, a partial result going in the direction of the equivalent of \cite{Bod05} for Potts model was obtained. We refer to the paper for details on open questions and conjectures. Bodineau's result being the key to our argument (not to mention the heavy use of the random-current representation, which itself is not available for the Potts model), we believe that the exponential decay of correlations would be even harder to obtain than the open problems mentioned in \cite{DumTas17}.
 
\paragraph{Organization} The paper is organized as follows. In the next section, we recall some background. In Section~3, we present the coarse-graining arguments relying on Pisztora's technique. In  Section~4, we prove Proposition~\ref{thm:mixing1}, conditionally on two technical statements which are proved in Section~5.

\paragraph{Acknowledgments} This research was supported by the ERC CriBLaM, an IDEX grant from Paris-Saclay, a grant from the Swiss FNS, and the NCCR SwissMAP. We thank Michael Aizenman for close collaboration during the beginning of this project. We thank Senya Shlosman for many useful comments. We are grateful to S\'{e}bastien Ott for pointing out the result by Alexander to us which led to the addition of Corollary~\ref{cor:mixing}. We also thank the anonymous referee for a careful review of the article and many useful comments. The paper was completed when AR was a Ph.D. student at IHES.

\section{Background}

\subsection{The FK-Ising model}

We will use a few properties of the FK-Ising model that we recall now. For details and proofs, we direct the reader to \cite{Gri06,Dum17}.

\paragraph{Spatial Markov property.} 

Let $H\subset G$ be two finite subgraphs of $\Z^d$ with respective edge-sets $E$ and $F$. A configuration $\omega$ on $G$ may be viewed as a configuration on $H$ by taking its restriction $\omega_{|E}$. 
The restriction of the configuration $\omega$ to edges of $F\setminus E$ induces boundary conditions on $G$. 
Namely, the spatial Markov property states that for any $p,q$ and any configuration $\xi$,
\begin{equation}\label{eq:domain Markov}
  \phi_{G,p}^\xi(\omega_{|E}=\cdot\,|\omega_e=\xi_e,\forall e\in F\setminus E)=\phi_{H,p}^{\xi}(\cdot).
\end{equation}
The spatial Markov property implies the following finite-energy property: for every $\xi$,
\begin{equation}\label{eq:finite energy}\tfrac{p}{2-p}\le \phi_{\{e\},p}^\xi[\omega_e]\le p.\end{equation}

\paragraph{Stochastic ordering for $q\ge1$.} 
For any finite graph $G$, the set $\{0,1\}^{E}$ has a natural partial order. 
An event $A$ is {\em increasing} if for every $\omega\le \omega'$, $\omega\in A$ implies $\omega'\in A$. 
The FK-Ising model satisfies the following properties. Fix $p\in[0,1]$ and $\xi\le \xi'$,
\begin{enumerate}
\item (FKG inequality) For every two increasing events $A$ and $B$,
\begin{equation}
\label{eq:FKG} \phi_{G,p}^\xi[A\cap B]\ge \phi_{G,p}^\xi[A]\phi_{G,p}^\xi[B].
\end{equation}
\item (comparison between boundary conditions) For every increasing event $A$, 
\begin{equation}
\label{eq:CBC}\phi_{G,p}^{\xi'}[A]\ge\phi_{G,p}^\xi[A].
\end{equation}\end{enumerate}

This last condition, together with \eqref{eq:domain Markov}, enables one to construct measures $\phi^1_{p}$ and $\phi^0_{p}$ in $\Z^d$ as weak limits of measures with free and wired boundary conditions in finite volume. It was proved in \cite{Bod06} (see also \cite{Rao17}) that $\phi^1_p=\phi^0_p$ for every $p\ne p_c$ (see \cite{AizDumSid15} for the case $p=p_c$), and this is the reason why we refer to the infinite-volume measure as simply $\phi_p$.

\begin{remark}\label{rmk:joint}
We will often consider couplings between two FK-Ising measures. In this case, we will use the spatial Markov property and the FKG inequality applied to {\em both configurations at the same time}, two properties that we will call {\em joint Markov property} and {\em joint FKG inequality}. These properties will be justified by the standard spatial Markov property and the FKG inequality combined with the special constructions of these measures. Since the justification is classical, we will omit it in this article.  \end{remark}

\paragraph{Edwards-Sokal coupling}

We will use the Edwards-Sokal coupling both in finite and infinite volume. On a finite graph,  the coupling goes as follows. Consider $\beta$ and $p$ related by $p=1-e^{-2\beta}$. Consider a configuration $\omega$ sampled according to $\phi^0_{G,p}$ and assign to each cluster $\mathcal C$ of $\omega$ a spin $\sigma_{\mathcal C}$ in $\{-1,+1\}$ uniformly and independently for each cluster. Then, set $\sigma_x=\sigma_{\mathcal C}$ for every $x\in\mathcal C$. 
As a direct consequence of this coupling, one obtains that 
\begin{equation}\label{eq:ES}
\langle\sigma_A\rangle_{G,\beta}=\phi_{G,p}^0[\mathcal F_A],
\end{equation}
where $\mathcal F_A$ is the event that every cluster of $\omega$ intersects an even number of times the set $A$. Note that when $A=\{x,y\}$, this translates into $\langle\sigma_x\sigma_y\rangle_{G,\beta}=\phi_{G,p}^0[x\leftrightarrow y]$.

We will also use the coupling in infinite volume. In this case, one can consider $\phi_p$ and assign a spin to each one of the finite clusters at random as explained previously, and a spin $+$ to the infinite clusters (there is in fact at most one such cluster). One then obtains the measure $\langle\cdot\rangle_{\beta}^+$. Altogether, we deduce from this representation that 
\begin{equation}
\label{eq:ESinfini}\langle\sigma_x\rangle_{\beta}^+=\phi_{p}[x\leftrightarrow\infty].
\end{equation}
We have in particular that $\beta_c=\tfrac12\log(1-p_c)$.

\paragraph{Griffiths inequality} The monotonicity properties of the FK-Ising model imply the following two classical inequalities, which will be very useful: for every sets of vertices $A$ and $B$,
\begin{equation}\label{eq:Griffiths1}
\langle\sigma_A\sigma_B\rangle_{G,\beta}^+\ge \langle\sigma_A\rangle_{G,\beta}^+\langle\sigma_B\rangle_{G,\beta}^+
\end{equation}
and, if exceptionally we consider the model with arbitrary coupling constants $J$ (we refer to $J$ in the notation by writing $\langle\cdot\rangle_{G,J,\beta}$ for the measure), we have that for every coupling constants $J\ge J'\ge0$,
\begin{equation}\label{eq:Griffiths2}
\langle\sigma_A\rangle_{G,J,\beta}\ge \langle\sigma_A\rangle_{G,J',\beta}.
\end{equation}
\subsection{The random-current representation}\label{subsec:random_current}
We will also use the random-current representation in several places. A {\em current}  configuration  $\n$ on a graph $G$ with vertex-set $V$ and edge-set $E$  is an integer valued function on $E$, i.e.~a function $\n:  E \mapsto \mathbb Z_+$.    A {\em source} of $\n=(\n(x,y):\{x,y\}\in E)$ is a vertex $x$ for which $\Delta_x(\n):=\sum_{y\in V:y\sim x}{\n}(x,y)$ is odd. The set of sources of $\n$ is denoted by $\partial\n$. The random current configuration's {\em weight}, at specified $\beta>0$, is given by
\begin{equation}
w_{\beta}(\n):=\prod_{\{x,y\}\in E}\frac{\beta^{\,{\n}(x,y)}}{{\n}(x,y)!}.
\end{equation}
For every finite subgraph $G$ of $\Z^d$, we also construct a graph $G^+=(V^+,E^+)$ with $V^+=V\cup\{\mathfrak g\}$, where $\mathfrak g$ is called the {\em ghost} vertex, and $E^+$ is the union of $E$ together with as many edges $\{x,\mathfrak g\}$ as edges between $x$ and a vertex of $\Z^d$ outside of $G$. Note that there can be multiple edges between two given vertices in $G^+$, but that only vertices on the boundary (i.e. the vertices neighboring a vertex in $\Z^d \setminus V$) of $G$ can be connected to the ghost vertex.

Correlations of the Ising model can be expressed in terms of the random-current representation via the following formula: for every $A\subset V$,
\begin{equation}
\langle\sigma_A\rangle_{G,\beta}:=\frac{\displaystyle\sum_{\n\in\Z_+^E:\partial\n=A}w_\beta(\n)}{\displaystyle\sum_{\n\in\Z_+^E:\partial\n=\emptyset}w_\beta(\n)}\qquad\text{and}\qquad\langle\sigma_A\rangle_{G,\beta}^+:=\frac{\displaystyle\sum_{\n\in\Z_+^{E^+}:\partial\n\cap V=A}w_\beta(\n)}{\displaystyle\sum_{\n\in\Z_+^{E^+}:\partial\n\cap V=\emptyset}w_\beta(\n)}.
\end{equation}
Note that the random-current representation of spin correlations for free boundary conditions involve currents on $G$, while those for plus boundary conditions involve currents on $G^+$.
 
The great utility of the random current representation results from a switching symmetry,  whose roots lie in a combinatorial identity of \cite{GHS}. Using this symmetry, the Ising phase transition was presented in \cite{Aiz82} as a phenomenon of percolation in a system of current loops.  Resulting relations have been instrumental in shedding light on the critical behavior of the model in various dimensions 
\cite{Aiz82,ABF87,AF86,AizDumSid15,AizDumTasWar18}.  
We do not wish to state the switching lemma here and refer to the corresponding literature. Rather, we present the applications we will need for our study. 

To express various correlation functions (of finite systems) in terms of probabilities for systems of currents with prescribed sources, we introduce the probability measure $\mathbb P^{A}_{G}$ on currents $\n\in \mathbb Z_+^E$ with $\partial\n=A$ by the formula
\begin{equation}
\mathbb P^{A}_{G}[\{\n\}]:=\frac{w_{\beta}(\n)}{\displaystyle\sum_{{\bf m}\in \Z_+^{E}:\partial{\bf m}=A}w_{\beta}({\bf m})}\,. 
\end{equation} 
Similarly, one defines $\mathbb P^A_{G^+}$ on currents $\n\in \mathbb Z_+^{E_+}$ with $\partial\n\cap V=A$.
Let $\mathbb P^{A,\emptyset}_{G^+,G}$ (resp.~$\mathbb P^{A,\emptyset}_{G^+,G^+}$) denote the law of two independent currents with respective laws $\mathbb P^A_{G^+}$ and $\mathbb P^\emptyset_{G}$ (resp.~$\mathbb P^\emptyset_{G^+}$).
The key relation of interest, which, at the risk of repeating ourselves, is a consequence of the switching lemma, is the following:
\begin{equation}\label{eq:switching1}
\mathbb P^{A,\emptyset}_{G^+,G}[\lr{G}{\mb n_1 + \mb n_2}{x}{y}]=\frac{\langle\sigma_A\sigma_x\sigma_y\rangle_{G}^+\langle\sigma_x\sigma_y\rangle_{G}}{\langle\sigma_A\rangle_{G}^+},
\end{equation}
where the event on the left denotes the fact that $x$ and $y$ are connected by a path $x=x_0\sim \dots \sim x_m=y$ of neighboring vertices of $G$ such that 
$(\n_1+\n_2)(x_i,x_{i+1})>0$ for every $0\le i<m$. Sometimes, we will consider two sets $X$ and $Y$ instead of $x$ and $y$. By this, we mean that some vertex in $X$ is connected to some vertex in $Y$. 

A special case of this relation consists in choosing $A=\{x,y\}$, which gives
\begin{equation}\label{eq:switching2}
\langle\sigma_x\sigma_y\rangle_{G, \beta}=\langle\sigma_x\sigma_y\rangle_{G, \beta}^+\,\mathbb P^{\{x,y\},\emptyset}_{G^+,G}[\lr{G}{\mb n_1 + \mb n_2}{x}{y}].
\end{equation}

We conclude the section by proving Lemma~\ref{lem:00}.
\begin{proof}[Proof of Lemma~\ref{lem:00}]
Let $G$ be a finite subgraph of $\Z^d$.  For $S \subset V^+$, let $\mathcal C_{\mb n}(S)$ denote the set of all vertices in $G^+$ which are connected to $S$ by $\mb n$. Since $A$ and $B$ are disjoint, the switching lemma implies that
$$
\langle \sigma_A\sigma_B \rangle^{+}_{G, \beta} - \langle \sigma_A \rangle^{+}_{G, \beta} \langle \sigma_B\rangle^{+}_{G, \beta} =  \langle \sigma_{A \cup B} \rangle^{+}_{G, \beta} \, \mathbb P^{A \cup B,\emptyset}_{G^+,G^+}[ 1 - \mathbb I_{\mathcal F_B}],$$
where $\mathcal F_B$, when $|B|$ is even, is the event that every connected component of $\mb n_1 + \mb n_2$ contains an even number of vertices from $B$. When $|B|$ is odd, $\mathcal F_B$ is the event that every connected component of $\mb n_1 + \mb n_2$ contains an even number of vertices from $B \cup \{ \mathfrak{g}\}$. Hence, to prove the lemma it suffices to demonstrate the inequality
\begin{align}\label{eq:goal12}
\sum_{\substack{\partial \mb n_1 \cap V =  A \cup B\\ \partial \mb n_2 = \emptyset}} \, &\omega_\beta(\mb n_1) \, \omega_\beta(\mb n_2) \, \mathbb I [\mb n_1 + \mb n_2 \notin \mathcal F_B]  \nonumber \\
&\leq \sum_{\substack{A' \subset A \\ | A'| \text{ is odd }}} \sum_{\substack{B' \subset B \\ | B'| \text{ is odd }}} \sum_{\substack{a \in A'\\ \, b \in B'}} \sum_{\substack{\partial \mb n_1 =  \{a ,b \}\\ \partial \mb n_2 = \emptyset}} \, \omega_\beta(\mb n_1) \, \omega_\beta(\mb n_2) \, \mathbb I\big [ \mb n_1 + \mb n_2 \notin \mathcal{F}_{\{b\}}\big].
\end{align}
To demonstrate \eqref{eq:goal12}, first notice that if $\mb n_1 + \mb n_2 \notin \mathcal F_B$, then there should be a connected component of $\mb n_1 + \mb n_2$ whose intersections with $A$ and $B$ are two sets $A'$ and $B'$ of odd cardinality, and furthermore this component should not contain $\mathfrak{g}$. The latter assumption is valid since if all such components contained $\mathfrak{g}$, then inevitably $\mb n_1 + \mb n_2 \in \mathcal{F}_B$. Therefore we have
\begin{align}
\label{eq:long_expr1}
&\sum_{\substack{\partial \mb n_1 =  A \cup B\\ \partial \mb n_2 = \emptyset}} \, \omega_\beta(\mb n_1) \, \omega_\beta(\mb n_2) \, \mathbb I [\mb n_1 + \mb n_2 \notin \mathcal F_B] \nonumber \\
&\quad \quad\leq \sum_{\substack{A' \subset A \\ | A'| \text{ is odd }}} \sum_{\substack{B' \subset B \\ | B'| \text{ is odd }}} \sum_{\substack{\partial \mb n_1 =  A \cup B\\ \partial \mb n_2 = \emptyset}} \, \omega_\beta(\mb n_1) \, \omega_\beta(\mb n_2) \, \mathbb I [ \, \mathcal C_{\mb n_1 + \mb n_2} (A')  \cap \big( A \cup B \cup \{\mathfrak{g}\} \big)= A' \cup B' ]\,.
\end{align}
We can write the third summation on the right hand side of the above display as
$$\sum_{S \in {\bf S}} \, \sum_{\substack{ \, \,  \mb n_1, \mb n_2 \in \N^{E_S} \\ \partial \mb n_1 =  A'\\ \partial \mb n_2 = B'}} \, \omega_\beta(\mb n_1) \, \omega_\beta(\mb n_2) \, \mathbb I  [ \mathcal C_{\mb n_1 + \mb n_2} (A') = S ]\times \Big( \langle \sigma_{A \cup B \setminus (A' \cup B')} \rangle_{G\setminus S} \sum_{\substack{\mb n_1, \mb n_2 \in \N ^{E\setminus {E_S}} \\ \partial \mb n_1 =\emptyset\\  \partial \mb n_2 = \emptyset}} \, \omega_\beta(\mb n_1) \, \omega_\beta(\mb n_2) \Big)$$
where $ {\bf S}$ is the set of all $S \subset V^+$ satisfying $S \cap ( A \cup B \cup \{\mathfrak{g}\} )= A' \cup B'$, and $E_S$ is the set of edges with at least one 
vertex in $S$. But this is bounded by 
\begin{align*}
\sum_{\substack{\partial \mb n_1 =  A'\\ \partial \mb n_2 = B'}} \, \omega_\beta(\mb n_1) \, \omega_\beta(\mb n_2) \, \mathbb I[ \, \mathfrak{g} \notin \mathcal C_{\mb n_1+\n_2} (A') = \mathcal C_{\mb n_1+\n_2} (B') ]\,,
\end{align*}
plugging which into \eqref{eq:long_expr1}, we get
\begin{align*}
&\sum_{\substack{\partial \mb n_1 =  A \cup B\\ \partial \mb n_2 = \emptyset}} \, \omega_\beta(\mb n_1) \, \omega_\beta(\mb n_2) \, \mathbb I [\mb n_1 + \mb n_2 \notin \mathcal F_B] \nonumber \\
&\quad \quad\leq \sum_{\substack{A' \subset A \\ | A'| \text{ is odd }}} \sum_{\substack{B' \subset B \\ | B'| \text{ is odd }}} \sum_{\substack{\partial \mb n_1 =  A'\\ \partial \mb n_2 = B'}} \, \omega_\beta(\mb n_1) \, \omega_\beta(\mb n_2) \, \mathbb I[ \, \mathfrak{g} \notin \mathcal C_{\mb n_1+\n_2} (A') = \mathcal C_{\mb n_1+\n_2} (B') ]\,.
\end{align*}
Now, for every $A' \subset A$ and $B' \subset B$,
\begin{align} \label{eq:qq11}
&\sum_{\substack{\partial \mb n_1 =  A' \cup B'\\ \partial \mb n_2 = \emptyset}} \, \omega_\beta(\mb n_1) \, \omega_\beta(\mb n_2) \, \mathbb I\big [ \mathfrak{g} \notin \mathcal C_{\mb n_1 + \mb n_2} (A') = \mathcal C_{\mb n_1 + \mb n_2} (B') \big] \nonumber \\
& \quad \quad \leq \sum_{\substack{\partial \mb n_1 =  A' \cup B'\\ \partial \mb n_2 = \emptyset}} \, \omega_\beta(\mb n_1) \, \omega_\beta(\mb n_2) \, \mathbb I\big [ \mathfrak{g} \notin \mathcal C_{\mb n_1 + \mb n_2} (A') \cup \mathcal C_{\mb n_1 + \mb n_2} (B') \big] \nonumber \\
 &\quad \quad \leq \sum_{S \in {\bf S'}} \,  \, \sum_{\substack{ \, \,  \mb n_1, \mb n_2 \in \N^{ E_S} \\ \partial \mb n_1 =  \emptyset\\ \partial \mb n_2 = \emptyset}} 
\, \omega_\beta(\mb n_1) \, \omega_\beta(\mb n_2) \, \mathbb I\big [ \mathcal{C}_{\mb n_1 +\mb n_2} (\mathfrak{g})  = S \big] 
\, \langle \sigma_{A' \cup B'} \rangle_{G\setminus S} \sum_{\substack{\mb n_1, \mb n_2 \in \N ^{E\setminus {E_S}} \\ \partial \mb n_1 = \emptyset \\ \partial \mb n_2 = \emptyset}} \, \omega_\beta(\mb n_1) \, \omega_\beta(\mb n_2),
\end{align}
where ${\bf S'}$ is the set of all $S \subset V^+$ such that $ S \cap (A' \cup B' \cup \{ \mathfrak{g}\}) = \{ \mathfrak{g}\}$. Since $A'$ and $B'$ are sets of odd cardinality, one can use the correlation inequality \begin{equation*}
\langle \sigma_{A' \cup B'} \rangle_{G \setminus S} \leq \sum_{\substack{a \in A'\\ b \in B'}}  \langle \sigma_{a} \sigma_{b} \rangle_{G \setminus S} \,
\end{equation*}
coming from the Edwards-Sokal coupling and the fact that $\mathcal F_{A'\cup B'}$ is included in the event that some $a\in A'$ is connected to some $b\in B'$. 
Substituting the above inequality in \eqref{eq:qq11} and resuming over $S\in {\bf S'}$ gives us
\begin{align*} 
\sum_{\substack{\partial \mb n_1 =  A'\\ \partial \mb n_2 = B'}} \, \omega_\beta(\mb n_1) \, \omega_\beta(\mb n_2) \, & \mathbb I\big [ \mathfrak{g} \notin \mathcal C_{\mb n_1 + \mb n_2} (A') = \mathcal C_{\mb n_1 + \mb n_2} (B') \big] \\ &\leq \sum_{\substack{a \in A' \\ b \in B'}}  
\sum_{\substack{\partial \mb n_1 =  \{a, b\}\\ \partial \mb n_2 = \emptyset}} \, \omega_\beta(\mb n_1) \, \omega_\beta(\mb n_2) \, \mathbb I\big [ \mathfrak{g} \notin \mathcal C_{\mb n_1 + \mb n_2} (a) = \mathcal C_{\mb n_1 + \mb n_2} (b) \big].
\end{align*}
Summing over $A'$ and $B'$ gives us \eqref{eq:goal12} and concludes the proof.
\end{proof}
\bigbreak
\centerline{\em In the rest of the article, $p$ and $\beta$ are fixed so that $1-p=e^{-2\beta}$ and dropped from the notation. }
 \section{Applications of Pisztora's coarse-grain approach}
 
The next subsection introduces the notion of good blocks. We then use a renormalization scheme to deduce that all big boxes are connected in FK-Ising. The last subsection derives Theorem~\ref{thm:mixing} from Proposition~\ref{thm:mixing1}.
 \subsection{Blocks and good blocks}
 
For $k\ge1$ and a set $S$, introduce the set $\mathcal B_k(S)$ of boxes $\L_{k}(x)\subset S$ with $x\in k\Z^d$.  From now on, we call an element of $\mathcal B_k(S)$ a {\em block} and often identify it with the set of its edges. 
Call a block $\mathbf B$ \emph{good} in $\omega$ if 
 \begin{itemize}[noitemsep]
 \item[(a)] $\omega_{|\mathbf B}$ contains a cluster touching all the $2d$ faces of $\mathbf B$;
 \item[(b)] any open path of length $k$ in $\mathbf B$ is included in this cluster.
\end{itemize}
 Boxes of this type were used by Pisztora to derive surface order large deviation 
estimates for the Ising, Potts and FK-percolation models. While the notion of good box there is slightly different, we refer to the definition of $U$ before \cite[Theorem~3.1]{Pis96} for a stronger notion than the notion above. 
The papers \cite{Pis96} and \cite{Bod05} together imply, for every $p>p_c$, the existence of $c>0$ such that for every $k$ and every boundary condition $\xi$,
\begin{equation}\label{eq:superconnected}\phi^\xi_{\Lambda_{2k}}[\L_{k} \text{ is good}] > 1 - \e^{-ck}.\end{equation}

\begin{lemma} \label{lem:boxcon}
For every $p>p_c$, there exists $c>0$ such that for every $n\le N\le e^{n^\alpha}$ (with $\alpha<d-1$) and every $x,y\in\Z^d$ such that $\L_n(x),\L_n(y)\subset\L_N$,
\begin{equation}
\phi^0_{\Lambda_N}[ \L_n(x) \longleftrightarrow \L_n(y)] \geq 1- \exp(-c n ^{d-1}).
\end{equation}
\end{lemma}
\begin{proof}
Fix $p<1$. Choose $\ep>0$ small enough (see later) and $k$ large enough that 
$$\phi^\xi_{\Lambda_{2k} }[\L_{k} \text{ is good}] > 1 - \ep$$
for every boundary condition $\xi$. Define a site percolation $\eta$ on $\mathcal B_k(\L_N)$ by saying that $\mathbf B$ is open if it is good, and closed otherwise. For a box $\mathbf B$, define $M(\mathbf B)$ to be the set of all the boxes in $\mathcal B_k(\L_N)$ whose centers are at a $\ell^\infty$ distance at most $3k$ of the center of ${\bf B}$. Note that, 
$$ \phi^{0}_{\Lambda_N} [ \eta_{\,\mathbf B} \, \vert  \, \eta_{\vert_{M(\mathbf B)^c}} ] \geq 1 -\varepsilon.$$ 
One deduces from \cite{LigSchSta97} that  the process $\eta$ dominates a Bernoulli percolation $\tilde\eta$ with parameter $p$ provided that $\ep$ is small enough.

A result of Deuschel and Pisztora \cite{deuschel1996surface} shows that for $p$ close enough to 1 and $n\ge1$, the probability that, for a fixed box $\L_n(x)$, there exists an open cluster in $\tilde\eta$ containing more than three fourths of the blocks in $\mathcal B_k(\L_n(x))$ is larger than $1-\exp[-2cn^{d-1}]$. The domination of $\tilde\eta$ by $\eta$ together with a union bound shows that this cluster also exists in $\eta$. Therefore, with probability 
$$1-|\L_N|\exp(-2cn^{d-1})\ge 1-|\Lambda_{\exp(n^\alpha)}|\exp(-2cn^{d-1})\ge 1-\exp(-cn^{d-1})$$
(for $n$ large enough), every box $\L_n(x)$ in $\L_N$ satisfies that there exists an open cluster in $\eta$ containing more than half the blocks contained in $\mathcal B_k(\L_n(x))$.

As a immediate consequence, on this event, all these clusters (meaning for every $\L_n(x)$) must in fact be the 
same cluster. In particular, there exist paths of pairwise neighboring good blocks going between any two boxes of 
size $n$ in $\Lambda_N$. Property~(a) of a good block implies that there must be an open path of length at least $k$ in the common region between two neighboring 
good blocks. Property~(b), on the other hand, implies that such a path must belong to the largest clusters of each of these blocks, i.e. the largest clusters of these blocks 
are connected. Together these give us that there exists an open path in $\omega$ between any two boxes of size $n$.
\end{proof}

\begin{remark}\label{rmk:a}The result of \cite{LigSchSta97} is quantitative. Since the probability of being good is larger than $1-\e^{-ck}$, the process $\eta$ dominates a Bernoulli percolation of parameter $p$ equal to $1-\e^{-c'k}$ for some $c'=c'(c)>0$ independent of $k$.
\end{remark}
\subsection{From Proposition~\ref{thm:mixing1} to Theorem~\ref{thm:mixing}}
\label{sec:mixing}

We now prove Theorem~\ref{thm:mixing} using Proposition~\ref{thm:mixing1}. Fix some boundary conditions $\xi$. We construct a coupling $\Phi^{\xi,1}_{\Lambda_n,}$ on pairs of configurations $(\omega^\xi,\omega^1)$ with $\phi^\xi_{\Lambda_n}$ being the law of the first marginal and $\phi^1_{\Lambda_n}$ the one of the second marginal. Call a block $\mathbf B\in\mathcal B_k(\L_n)$ {\em very good} in $(\omega^\xi,\omega^1)$ if it is good in $\omega^\xi_{|\mathbf B}$ and $\omega^\xi_{|\mathbf B}=\omega^1_{|\mathbf B}$.

We construct algorithmically a coupling $\Phi^{\xi,1}_{\Lambda_n}$ block by block following steps indexed by $t$. We assume that $k$ divides $n$ (the construction can be adapted in a trivial fashion to the general case). Below, $C_t$ will denote the set of edges $e$ for which $(\omega^\xi_e,\omega^1_e)$ is sampled before step $t$. The set $A_t$ is the set of blocks that have been sampled up to time $t$.  
Set 
$$\begin{cases}A_0:=\emptyset,\\ 
B _ 0:=\{\mathbf B\in\mathcal B_k(\L_n),\mathbf B\cap(\L_{n}\setminus\L_{n-1})\ne \emptyset\} ,\\ C_0 := \emptyset.\end{cases}$$ 
(The set $B_0$ corresponds to blocks adjacent to the boundary of $\L_n$.) At Step $t$, the algorithm proceeds as follows: 
\begin{itemize}
\item If $B_t=\emptyset$,  sample $\omega^\xi_{|E_n\setminus C_{t}} = \omega ^1_{|E_n\setminus C_{t}}$ according to the measure $\phi^\xi_{\Lambda_n} (  \, . \, \vert \omega^\xi_{|C_{t}})$ and terminate the algorithm.

\item If $B_t\ne \emptyset$, choose a block $\mathbf B\in B_{t}$ and define $D_t :=\mathbf B \setminus C_t$.
Then, sample $\omega^\xi_{|D_t}\le \omega^1 _{|D_t}$ such that $\omega^\xi_{|D_t}$ has the law $\phi^\xi_{\Lambda_n} [\cdot_{|D_t} \vert\omega^\xi_{|C_t}]$ and $\omega^1_{|D_t}$ the law $\phi^1_{\Lambda_n} [  \cdot_{|D_t} \vert \omega^1_{|C_t}]$. Finally, if $N(\mathbf B)$ denotes the set of blocks in $\mathcal B_k(\L_n)$ that intersect $\mathbf B$, set 
$$\begin{cases}
A_{t+1} := A_t \cup \{\mathbf B\},\\
B_{t+1} := B_t \setminus \{\mathbf B\}\text{ if }\mathbf B\text{ is very good and $B_t\cup N(\mathbf B)\setminus A_{t+1}$ otherwise,}\\
C_{t+1} := C_t \cup D_t.
\end{cases}$$
\end{itemize}
When the algorithm terminates, set $T$ for the terminal time and return  $(\omega^\xi,\omega^1)$.

\begin{lemma}\label{lem:a1}
The measure $\Phi^{\xi,1}_{\L_n}$ satisfies that 
\begin{itemize}[noitemsep]
\item $\omega^\xi$ has law $\phi^\xi_{\L_n}$ and $\omega^1$ has law $\phi^1_{\L_n}$,
\item $\omega^\xi\le \omega^1$,
\item $\omega^\xi_e=\omega^1_e$ if $e\notin C_{T}$.
\end{itemize} 
\end{lemma}
\begin{proof}
To prove this lemma, it suffices to prove that the boundary conditions induced on $E_n\setminus C_T$ by $\omega^\xi_{|C_T}$ and $\omega^1_{|C_T}$ are the same. Indeed, the three items follow readily from this observation and the joint spatial Markov property. 
Since $\omega^\xi_{|C_T}\le\omega^1_{|C_T}$, it suffices to prove that for every vertices $x$ and $y$ on the boundary of $E_n\setminus C_T$,
\begin{equation}\label{eq:js}x\stackrel{\omega^1_{|C_T}}{\longleftrightarrow}y\ \ \Longrightarrow \ \ x\stackrel{\omega^\xi_{|C_T}}{\longleftrightarrow}y\quad.\end{equation}
If $C_T=E_n$ there is nothing to prove. If $C_T\ne E_n$, let 
$$Y:=\mathcal B_k(\L_n)\setminus A_T\qquad\text{ and 
}\qquad Z:=\{\mathbf B\in A_T\text{ intersecting a block in }Y\},$$ where below we also identify $Z$ with the set of edges in its blocks. 
The fact that $B_T= \emptyset$ implies that $\omega^\xi$ and $\omega^1$ coincide on $Z$ and every $\mathbf B\in Z$ is very good for $\omega^1$ and $\omega^\xi$. Now, for every connected component $C$ of $C_T$, the set of blocks in $Z$ that intersect $C$ are connected in the following sense: every two such blocks ${\bf B}$ and ${\bf B'}$ are connected by a sequence of blocks ${\bf B}={\bf B}_1,\dots,{\bf B}_s={\bf B'}$ of $Z$ such that ${\bf B}_i\cap{\bf B}_{i+1}\ne \emptyset$ for every $1\le i<s$. The discussion in the proof of Lemma~\ref{lem:boxcon} implies that all the big clusters in these blocks are connected to each other inside $Z$.  Thus, if two vertices $x$ and $y$ on the boundary of $C_T$ are connected in $\omega^\xi_{|C_T}$, then they already are in $\omega^\xi_{|Z}=\omega^1_{|Z}$. This proves \eqref{eq:js} and therefore concludes the proof.\end{proof}
\begin{proof}[Proof of Theorem~\ref{thm:mixing}]
Let $D$ be the number of blocks in $M({\bf B})$ (recall the definition of $M({\bf B})$ from the proof of Lemma~\ref{lem:boxcon} and observe that the choice of ${\bf B}$ is irrelevant here), and choose $\ep$ so that $2D\ep<1/e$. By Proposition~\ref{thm:mixing1}, pick $k$ large enough that
\begin{align*}
\sum_{e\in\L_{k}}\phi^1_{\L_{2k}}[\omega_e]-\phi^0_{\L_{2k}}[\omega_e]\le |\L_{k}|e^{-c(\log k)^{1+c}}\le \varepsilon.
\end{align*}
Also, assume that $k$ is chosen large enough that the probability of being good is larger than $1-\ep$ (this is doable by using \eqref{eq:superconnected}).

 By the Markov property \eqref{eq:domain Markov}, \eqref{eq:main2} follows from the next claim:  for every boundary condition $\xi$ on $\Z^d$ and every event $A$ depending on edges in $E_{n/2}$ only,  
\begin{equation}\label{eq:bk}|\phi^\xi_{\L_n}[A]-\phi^1_{\L_n}[A]|\le \Phi^{\xi,1}[C_T\cap E_{n/2}\ne \emptyset]\le\exp[-cn].\end{equation}
The first inequality follows directly from the third item of Lemma~\ref{lem:a1}. We therefore focus on the second one.
Assume that $C_T\cap E_{n/2}\ne \emptyset$. Then, there must exist a sequence of times $t_1<t_2<\dots<t_s$ with $s\ge n/(16k)$ such that the blocks $\mathbf B_1,\dots,\mathbf B_s$ used by the algorithm constructing $\Phi^{\xi,1}_{\L_n}$ at times $t_1,\dots,t_s$ are disjoint, not very good, and such that ${\bf B}_{i+1}\in M({\bf B}_i)$ for every $1\le i<s$. The choice of $k$ immediately implies that conditioned on $\omega^\xi$ and $\omega^1$ in the blocks $\mathbf B_1,\dots,\mathbf B_i$, the block $\mathbf B_{i+1}$ is not very good with probability smaller than or equal to $2\ep$. We deduce that 
\begin{equation} \label{eq:goal2}
\Phi^{\xi,1}_{\Lambda_n} [ C_T\cap E_{n/2}\ne\emptyset ] \le(D2\ep)^{n/(16k)}.
\end{equation}  
The choice of $\ep$ proves \eqref{eq:bk} with $c=1/(16k)$. 
\end{proof}
\section{Proof of Proposition~\ref{thm:mixing1}}
The proof of Proposition~\ref{thm:mixing1} can be decomposed into several steps. We first prove a very weak mixing property (where the Radon-Nikodym derivative is bounded from above by $\exp(cn^{d-1})$. We then use this property to show Proposition~\ref{thm:mixing1}, but basing our study on two technical lemmata whose proofs are postponed to the next section.

\subsection{Connection probability between boxes for double random current}
The main result of this section is the proposition below. The proof relies on a number of properties of the random current, combined with a mixing property for the Fortuin-Kasteleyn percolation.
 \begin{proposition}
 	\label{prop:connectivity}
 	Fix $\beta>\beta_c$. There exists $c>0$ such that for every $N \geq 2n$, every $x, y,w,z \in \L_{N}$ with $w, z \notin\L_n(x) \cup \L_n(y)\subset\L_N$, we have
 	$$\P^{\{w, z\}, \emptyset}_{\L_N^+, \L_N}[\lr{\L_N}{\mb n_1 + \mb n_2}{\L_n(x)}{\L_n(y)}] \geq 1 - \e^{-cn^{d-1}}\,.$$
 \end{proposition}
 In the whole section, we will use the following notation. Without loss of generality, we can assume that $n$ is even. Assume that $\L_n(x)\cap\L_n(y)=\emptyset$. Construct the graphs 
 $\L_N^{\bullet\bullet}:=\L_N^{\bullet\bullet}(x,y,n)$ and  $(\L_N^{\bullet\bullet})^+:=(\L_N^{\bullet\bullet}(x,y,n))^+$ obtained from $\L_N$ and 
 $\L_N^+$ by replacing $\L_{n/2}(x)$ and $\L_{n/2}(y)$ by two vertices ${\bf x}$ and ${\bf y}$, which are connected to each vertex outside $\L_{n/2}(x)\cup
 \L_{n/2}(y)$ by the number of edges between this vertex and $\L_{n/2}(x)$ (resp.~$\L_{n/2}(y)$). Similarly, one defines $\L_{n}^\bullet(x)$ to be the graph obtained from $\L_{n}(x)$ by merging all the vertices in $\L_{{n/2}}(x)$ following the same procedure as for $\L_N^{\bullet\bullet}$. 
 Note that the FK-Ising and Ising models on those graphs can be seen as models on the original graphs, for which edges $ e $ in $\L_{n/2}(x) \cup \L_{n/2}(y)$ have $p_e=1$ (for the FK-Ising) and have infinite coupling constants (for the Ising model).
  This observation is useful to keep in mind when applying \eqref{eq:Griffiths2} for instance.
 
\begin{proof}If $\L_n(x)\cap\L_n(y)\ne\emptyset$, one does not need to do anything. We therefore assume that $\L_n(x)\cap\L_n(y)=\emptyset$ and consider the graph $\L_N^{\bullet\bullet}$ defined above.
Equation~\eqref{eq:switching1} implies that 
\begin{align}
	\label{eq:collapsed_dc1}
	\P^{\{w, z\}, \emptyset}_{(\L_N^{\bullet\bullet})^+, \L_N^{\bullet\bullet}}[\lr{\L_N}{\mb n_1 + \mb n_2}{{\bf x}}{{\bf y}}]&=\frac{\ising{\sigma_{{\bf x}}\sigma_{{\bf y}}\sigma_w\sigma_z}^+_{\L_N^{\bullet\bullet}}\,\ising{\sigma_{{\bf x}}\sigma_{{\bf y}}}_{\L_N^{\bullet\bullet}}}{\ising{\sigma_w\sigma_z}^+_{\L_N^{\bullet\bullet}}}\ge \ising{\sigma_{{\bf x}}\sigma_{{\bf y}}}_{\L_N^{\bullet\bullet}}^2,
	\end{align}
	where in the second inequality we used Griffiths' inequalities \eqref{eq:Griffiths1} and \eqref{eq:Griffiths2}. The Edwards-Sokal coupling \eqref{eq:ES} and the FKG inequality \eqref{eq:FKG} (the measure on $\L_N^{\bullet\bullet}$ can be understood as the measure on $\L_N$ with edges in $\L_{n/2}(x)\cup\L_{n/2}(y)$ conditioned to be open) imply that 
	\begin{equation}\label{eq:k}\ising{\sigma_{{\bf x}}\sigma_{{\bf y}}}_{\L_N^{\bullet\bullet}}=\phi_{\L_N^{\bullet\bullet}}[{\bf x}\leftrightarrow{\bf y}]\ge \phi_{\L_N}[\L_{{n/2}}(x)\leftrightarrow\L_{{n/2}}(y)]\ge 1-\exp(-cn^{d-1}),\end{equation}
	where the last inequality follows from Lemma~\ref{lem:boxcon}.
	
	Assume for a moment that for every $\ep>0$, one can choose $n$ large enough that for every choice of $N$ and $x,y$, for every event $\mathcal E$ depending on $(\n_1,\n_2)$ on edges in $E:=E_N\setminus (E_{n}(x)\cup E_{n}(y))$ only, we have that
	\begin{equation}\label{eq:kk}\P^{\{w, z\}, \emptyset}_{\L_N^+, \L_N}[\mathcal E]\le e^{\ep n^{d-1}}\P^{\{w, z\}, \emptyset}_{(\L_N^{\bullet\bullet})^+, \L_N^{\bullet\bullet}}[\mathcal E].\end{equation}
	Then, one would deduce from plugging the estimate \eqref{eq:k} into \eqref{eq:collapsed_dc1}, and then using \eqref{eq:kk} (together with the fact that if ${\bf x}$ is connected to ${\bf y}$, then $\L_n(x)$ is connected to $\L_n(y)$), that 
	$$\P^{\{w, z\}, \emptyset}_{\L_N^+, \L_N}[\nlr{\L_N}{\mb n_1 + \mb n_2}{\L_{n}(x)}{\L_{n}(y)}]\le e^{\ep n^{d-1}}\P^{\{w, z\}, \emptyset}_{(\L_N^{\bullet\bullet})^+, \L_N^{\bullet\bullet}}[\nlr{\L_N}{\mb n_1 + \mb n_2}{\L_{n}(x)}{\L_{n}(y)}]\le e^{(\ep-2c) n^{d-1}},$$
	which would conclude the proof.
We therefore focus on proving \eqref{eq:kk}. Let $\L \coloneqq \L_n(x) \cup \L_n(y)$ and consider a current ${\bf m}\in \N^E$ with $\partial {\bf m}\subset\partial\Lambda$. Let us also consider the event $\mathcal E_{\bf m}$ 
that $\n$ is equal to ${\bf m}$ on $E$. A trivial manipulation using the expression of 
weights for currents (in the first line) and then Griffiths inequality \eqref{eq:Griffiths2} in the second give that 
\begin{align*}
\frac{\P^{\emptyset}_{\L_N}[\mathcal E_{\bf m}]}{\P^{\emptyset}_{\L_N^{\bullet\bullet}}[\mathcal E_{\bf m}]}&=\frac{\ising{\sigma_{\partial{\bf m}}}_{\L_{n}(x)\cup\L_{n}(y)}}{\ising{\sigma_{\partial{\bf m}}}_{\L_{n}^\bullet(x)\cup\L_{n}^\bullet(y)}}\,\frac{\displaystyle\sum_{\mb m' \in  \N^E, \, \partial\mb m'\subset\partial\L}w_\beta(\mb m')\ising{\sigma_{\partial\mb m'}}_{\L_{n}^\bullet(x)\cup\L_{n}^\bullet(y)}}{\displaystyle\sum_{\mb m' \in  \N^E, \, \partial\mb m'\subset\partial\L}w_\beta(\mb m')\ising{\sigma_{\partial\mb m'}}_{\L_{n}(x)\cup\L_{n}(y)}}\\
&\le \max_{\substack{A\subset\partial\L_{n}(x)\\ |A|\text{ even}}}\frac{\ising{\sigma_A}_{\L_{n}^\bullet(x)}}{\ising{\sigma_A}_{\L_{n}(x)}}\times\max_{\substack{B\subset\partial\L_{n}(y)\\ |B|\text{ even}}}\frac{\ising{\sigma_B}_{\L_{n}^\bullet(y)}}{\ising{\sigma_B}_{\L_{n}(y)}},
\end{align*}
where the sums on the first line are on currents on $E$ (with the right sources). 
A similar identity can be derived for $\P^{\{z,w\}}_{\L_N^+}$ and $\P^{\{z,w\}}_{(\L_N^{\bullet\bullet})^+}$ instead of $\P^{\emptyset}_{\L_N}$ and $\P^{\emptyset}_{\L_N^{\bullet\bullet}}$. 

We deduce, by decomposing on possible values of $(\n_1,\n_2)$ on $E$, that 
$$\P^{\{w, z\}, \emptyset}_{\L_N^+, \L_N}[\mathcal E]\le  \Big(\max_{\substack{A\subset\partial\L_{n}\\ |A|\text{ even}}}\frac{\ising{\sigma_A}_{\L_{n}^\bullet}}{\ising{\sigma_A}_{\L_{n}}}\Big)^4\ \P^{\{w, z\}, \emptyset}_{(\L_N^{\bullet\bullet})^+, \L_N^{\bullet\bullet}}[\mathcal E],$$
so that \eqref{eq:kk} follows from the following lemma.
\end{proof}

\begin{lemma}\label{lem:weak weak mix}
Fix $\beta>\beta_c$ and $\ep>0$. Then, for $n$ large enough, we have that for every $A\subset\partial\L_n$,
\begin{equation}\label{eq:wwmixing}\ising{\sigma_A}_{\L_n^\bullet} \leq \e^{\ep n^{d-1}}\ising{\sigma_A}_{\L_n}.\end{equation}
\end{lemma}

\begin{proof} Fix $\beta>\beta_c$ and set $p=1-e^{-2\beta}>p_c$. By \eqref{eq:ES}, we must prove that $\phi_{\L_n^\bullet}^0[\mathcal F_A]\le\e^{\ep n^{d-1}}\phi_{\L_n}^0[\mathcal F_A]$. Note that since $\ep$ is arbitrary and since $\mathcal F_A$ occurs if all the edges on the boundary of $\L_n$ are open (and therefore has a probability larger than $(\tfrac p{2-p})^{d|\partial\L_n|}$), it suffices to show that 
\begin{equation}\label{eq:aa}\phi_{\L_n^\bullet}^0[\mathcal F_A]\le  \e^{-cn^d}+n\,\big(\tfrac{2-p}p\big)^{\ep n^{d-1}}\phi_{\L_n}^0[\mathcal F_A].\end{equation}
We do this by constructing a coupling $\Phi$ on pairs $(\omega,\omega^\bullet)$ of percolation configurations on $\{0,1\}^{E_n}$ for which $\omega\le \omega^\bullet$ have respective laws $\phi_{\L_n}^0$ and  $\phi_{\L_n^\bullet}^0$ (for consistency, we see a configuration on $\L_n^\bullet$ as a configuration on $E_n$ for which edges in $E_{n/2}$ are automatically open). 

Consider $k$  to be fixed later and assume that $k$ divides $n/2$  (one can trivially adapt the proof if $k$ does not). Also, set $T:=n/(4k)$. A block $\mathbf B$ is called {\em bad} if it is either not good or if $\omega_{|\mathbf B}\ne \omega^\bullet_{|\mathbf B}$.
For every $t\in[T,2T]$, define the three sets $C_t:=E_{2k t}$, $D_t$  the set of edges in $C_{t + 1}\setminus C_t$ within a distance of $2k$ of a bad block included in $C_t$, and $D'_t:=C_{t+1}\setminus (C_t\cup D_t)$; see Fig.~\ref{fig:Hs}. Note that $C_t$ is a deterministic set, but $D_t$ and $D'_t$ are random variables which are measurable in terms of $(\omega_{|C_t},\omega^\bullet_{|C_t})$.

Then, the coupling is constructed as follows. First, sample $\omega_{|E_{n/2}}$ according to $\phi_{\L_n}(\cdot_{|E_{n/2}})$ (recall that the edges of $\omega^\bullet$ are necessarily open in $E_{n/2}$). Then, for every $t\ge T$, 
\begin{itemize}[noitemsep]
\item Sample $\omega_{|D_t}\le\omega^\bullet_{|D_t}$ according to $\phi_{\L_n}^0(\cdot_{|D_t}|\omega_{|C_t})$ and $\phi^0_{\L_n^\bullet}(\cdot_{|D_t}|\omega^\bullet_{|C_t})$,
\item Sample $\omega_{|D'_t}\le\omega^\bullet_{|D'_t}$ according to $\phi_{\L_n}^0(\cdot_{|D'_t}|\omega_{|C_t\cup D_t})$ and $\phi^0_{\L_n^\bullet}(\cdot_{|D'_t}|\omega^\bullet_{|C_t\cup D_t})$.
\end{itemize}

For $t\in[T,2T]$, we call $\mathcal G_t$ the event that there are fewer than $\ep n^{d-1} $ edges in $D_t$.
The inequality \eqref{eq:aa} follows readily from the following two claims.
\paragraph{Claim 1.} {\em For every $t\ge T$, we have that }
$$\Phi[\{\omega^\bullet\in\mathcal F_A\}\cap\mathcal G_t]\le \big(\tfrac{2-p}p\big)^{\ep n^{d-1}}\Phi[\omega\in\mathcal F_A].$$

\paragraph{Claim 2.} {\em There exist $k$ and $c>0$ such that for every $n$ large enough,}
$$\Phi\big[\bigcap_{t\ge T}\mathcal G_t^c\big]\le \e^{-cn^d}.$$

To conclude, we only need to prove those claims.
\bigbreak\noindent{\em Proof of Claim 1.}
 \begin{figure}[!htb]
   \centering
    \includegraphics[width=.60\linewidth]{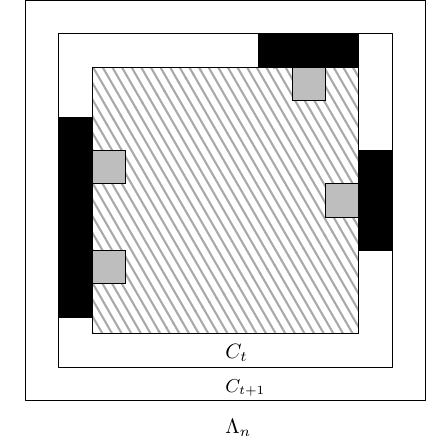}
    \caption{{\bf A realization of the event $\mathcal H_t$.} 
    The grey boxes are the bad blocks in $C_t$ that are adjacent to $C_{t + 1} \setminus C_t$. The dark rectangles represent the edges in $D_t$.}
    \label{fig:Hs}
  \end{figure} 
Let $\mathcal H_t$ be the event that every edge in $D_t$ is open. We refer the reader to Figure~\ref{fig:Hs} for an illustration. 
 On $\mathcal H_t\cap\mathcal G_t$, the boundary conditions on $E_n\setminus(C_t\cup D_t)$ induced by $\omega$ and $\omega^\bullet$ are the same. Since $\mathcal H_t\cap\mathcal G_t$ depends on edges in $C_t\cup D_t$ only, the joint Markov property of the coupling implies that $\omega_{|E_n\setminus (C_t\cup D_t)}=\omega^\bullet_{|E_n\setminus (C_t\cup D_t)}$ and $\omega^\bullet\in \mathcal F_A$ if and only if $\omega\in\mathcal F_A$. We deduce that 
\begin{equation}\label{eq:uh}\Phi[\{\omega^\bullet\in\mathcal F_A\}\cap\mathcal H_t\cap\mathcal G_t]=\Phi[\{\omega\in\mathcal F_A\}\cap\mathcal H_t\cap\mathcal G_t]\le\Phi[\omega\in\mathcal F_A].\end{equation}
Conditioned on $(\omega,\omega^\bullet)_{|C_t}$, the joint FKG inequality and the joint Markov  property  imply that 
\begin{equation}\label{eq:yyh}\Phi[\{\omega^\bullet\in\mathcal F_A\}\cap \mathcal H_t|(\omega,\omega^\bullet)_{|C_t}]\ge\Phi[\{\omega^\bullet\in\mathcal F_A\}|(\omega,\omega^\bullet)_{|C_t}]\Phi[\mathcal H_t|(\omega,\omega^\bullet)_{|C_t}].\end{equation}
If $(\omega,\omega^\bullet)_{|C_t}\in\mathcal G_t$, at most $\ep n^{d-1}$ edges must be open for $\mathcal H_t$ to occur.  We deduce that 
$$\Phi[\mathcal H_t|(\omega,\omega^\bullet)_{|C_t}]\ge \big(\tfrac{p}{2-p}\big)^{\ep n^{d-1}}.$$
Since $\mathcal G_t$ depends on edges in $C_t$ only, we may insert the previous estimate in \eqref{eq:yyh} and then integrate on $(\omega,\omega^\bullet)_{|C_t}\in\mathcal G_t$ to obtain that
$$\Phi[\{\omega^\bullet\in\mathcal F_A\}\cap \mathcal H_t\cap\mathcal G_t]\ge\Phi[\{\omega^\bullet\in\mathcal F_A\}\cap\mathcal G_t] \big(\tfrac{p}{2-p}\big)^{\ep n^{d-1}}.$$
Putting this inequality into \eqref{eq:uh} concludes the proof of Claim~1.
\bigbreak
\noindent{\em Proof of Claim 2.}
For no $\mathcal G_t$ to occur, there must be at least $\ep n^{d-1}/|E_{3k}|\times n/(4k)=:\ep_k|E_n|$ bad blocks. As a consequence, one of the following three things must happen:
\begin{itemize}[noitemsep,nolistsep]
\item there are more than $(\ep_k/2)|E_n|$ blocks in $\mathcal B_k(\L_n)$ that are not good.
\item the number of open edges in $\omega^\bullet_{|E_n\setminus E_{n/2}}$ is larger than $(\phi[\omega_e]+\ep_k/4)|E_n\setminus E_{n/2}|$.
\item the number of open edges in $\omega_{|E_n\setminus E_{n/2}}$ is smaller than $(\phi[\omega_e]-\ep_k/4)|E_n\setminus E_{n/2}|$.
\end{itemize}
Note that these three events are involving either $\omega$ or $\omega^\bullet$, so that we can now ignore the coupling $\Phi$.
We bound the probability of each one of these events separately.

For the first item, Remark~\ref{rmk:a} enables us to choose $k$ large enough that the process of good boxes dominates a Bernoulli percolation of parameter $p>1-\ep_k$. We deduce from large deviations for iid Bernoulli variables that the probability of the event decays as $e^{-cn^d}$ uniformly in $n$, where $c=c(k,\ep)>0$.

For the second item, the uniqueness of the infinite-volume measure $\phi$ implies that one may choose $K=K(k,\ep)$ large enough that 
$$\phi^1_{\L_K}\Big[\sum_{e\in E_K}\omega_e\Big]< (\phi[\omega_e]+\ep_k/4)|E_K|.$$
Now, consider a family of balls of size $K$ covering $\L_n$ and with disjoint interiors. By sampling the FK-Ising measure $\phi^1_{\L_K}$ ball by ball (here, the order in which the balls are sampled is irrelevant), the comparison between boundary conditions \eqref{eq:CBC} enables us to compare the number of open edges in $E_n\setminus E_{n/2}$ to a sum of independent random variables. The theory of large deviations implies that the probability of the second event is also bounded by $\exp(-cn^d)$ uniformly in $n$, where $c=c(K,k,\ep)>0$.

The third item follows from the same reasoning as the second one (except it does not involve the uniqueness of the infinite-volume measure).
\end{proof}

\subsection{Proof of Proposition~\ref{thm:mixing1}}

In the rest of the paper, for a current $\mb n$ on $\L_N$ or $\L_N^+$, we use $\mathcal C_{\mb n}(S)$ to denote the set of all vertices \emph{in} $\L_N$ which are connected to $S$ by $\mb n$. Notice that this is slightly different from the convention in Lemma~\ref{lem:00} where we also included $\mathfrak g$ in $\mathcal C_{\mb n}(S)$. 
\medbreak
From now on, we fix $N=N(n): = e ^{n^ \alpha}$ where $1 <\alpha < d-1$, and $ e =\{x,y\}$. Using \eqref{eq:ES} and a simple relation between the probability of $\omega_e=1$ and the probability that $x$ and $y$ are connected together, we find that 
\begin{align}
\label{eq:two_arms_bnd}
\phi_{\L_N}^1[\omega_e]-\phi_{\L_N}^0[\omega_e]  &\stackrel{}{=}\tfrac{p}2(\ising{\sigma_x \sigma_y}^+_{\L_N} -\ising{\sigma_x \sigma_y}_{\L_N}) \stackrel{\eqref{eq:switching2}}{\leq} \tfrac{p}2\ising{\sigma_x \sigma_y}^+_{\L_N}\,\P^{\{x, y\}, \emptyset}_{\L_N^+, \L_N}[\nlr{\L_N}{\mb n_1 + \mb n_2}{x}{y}].
\end{align}
Notice that $x$ and $y$ are the sources of $\mb n_1$ and as such they must be connected to 
each other by $\mb n_1$ in $\L_N^+$. Thus, when $x$ is not connected to $y$ in $\L_N$, there must be a path  between $x$ and $\partial \L_N$ in $\mathcal C_{\mb n_1}(x)$ on one side, and 
between $y$ and $ \partial \L_N$ in $\mathcal 
C_{\mb n_1}(y)$ on the other side. Since $\{x, y\} \in E_{N/4}$, the path emanating from $y$ must intersect at least $\tfrac12N/n$ 
blocks in $\mathcal B_n(\Lambda_{N})$. Among these blocks, there is always either $\tfrac14N/n$ many intersecting $\mathcal C_{\mb n_1 + \mb n_2}(x)$ or $\tfrac14N/n$ many not intersecting it. 
Repeating a similar argument for the path between $x$ and $\partial \L_N$, we obtain that if 
$x$ is not connected to $y$ in $\n_1+\n_2$ and we define the following events,
\begin{align*}
\mathcal C_{x,y} :=& \big\{\mbox{there exist } \tfrac14N/n\mbox{ blocks } {\bf B}\in \mathcal B_n(\Lambda_{N})\mbox{ such that }\mathcal C_{\mb n_1 + \mb n_2}(x) \cap {\bf B} = \emptyset\mbox{ and } \mathcal C_{\mb n_1}(y) \cap{\bf B}\\
& \mbox{ contains a vertex having at least two neighbors in }\mathcal C_{\mb n_1}(y)\big\},\\
\mathcal B_{x,y} :=& \{\mbox{there exist } \tfrac14N/n\mbox{ blocks } {\bf B}\in \mathcal B_n(\Lambda_{N})\mbox{ such that }\mathcal C_{\mb n_1 + \mb n_2}(x) \cap {\bf B} \ne \emptyset\mbox{ and } \mathcal C_{\mb n_1}(y) \cap{\bf B}\\
& \mbox{ contains a vertex having at least two neighbors in }\mathcal C_{\mb n_1}(y)\},
\end{align*}
then, either $\mathcal B_{x,y}$, $\mathcal B_{y,x}$ or $\mathcal C := \mathcal C_{x, y} \cap \mathcal C_{y, x}$ must occur. 

Therefore, it suffices to show that the intersection of the event on the right hand side of \eqref{eq:two_arms_bnd} with either $\mathcal C$, $\mathcal B_{x, y}$ or $\mathcal B_{y,x}$ has very small probability. This is the subject of the two lemmata below which are proved in the next section. 
\begin{lemma}
	\label{lem:gluing3}
	There exists $c>0$ such that for every $n$ large enough such that $\{x,y\}\in E_{N/4}$,
	$$\P^{\{x, y\}, \emptyset}_{{\L_N}^+, {\L_N}} [\mathcal C \cap \{\nlr{\L_N}{\mb n_1 + \mb n_2}{x}{y}\}] \leq \e^{-cn^{d-1}}\,.$$
\end{lemma}

\begin{lemma}
\label{lem:gluing2}
There exists $c>0$ such that for every $n$ large enough such that $\{x,y\}\in E_{N/4}$,
$$\P^{\{x,y\}, \emptyset}_{\L_N^+, \L_N}[\mathcal B_{x,y} \cap \{\nlr{\L_N}{\mb n_1 + \mb n_2}{x}{y}\}] \leq \e^{-cn^{d-1}}\,.$$
\end{lemma}
Note that the bounds given below imply Proposition~\ref{thm:mixing1} since $N={\rm e} ^{n^ \alpha}$ gives that the bounds on the right are of the form $\exp[-(\log N)^{(d-1)/\alpha}]$, and that $\alpha<d-1$.

\section{Proofs of Lemmata~\ref{lem:gluing3} and \ref{lem:gluing2}}
In the remainder of the paper, we will use $C$ and $c$ to denote finite, positive constants depending on at most $\beta$ and $d$. The values of these constants 
may vary from one line to the next. It might be helpful to think of $C$ and $c$ as large and small
positive numbers respectively.

In order to prove these lemmata, we will use a \emph{multi-valued map principle}.
\begin{lemma}
\label{lem:mvmp}
Consider a probability space $(\mathcal S, \mathfrak P(\mathcal S), \mu)$ where $\mathcal S$ is at most countable and $\mathfrak P(\mathcal S)$ is the set of all subsets of 
$\mathcal S$\,. Let $A, B \subset \mathcal S$ and $\mathcal R \subset A \times B$ be a relation satisfying the following two properties:
\begin{itemize}[noitemsep]
\item[(i)] $|\mathcal R(s)| \geq K$ for every $s \in A$
where $\mathcal R(s) := \{s' \in B:(s, s') \in \mathcal R  \}$.
\item[(ii)] $\sum_{s \in \mathcal R^{-1}(s')} \mu(s)\leq k \mu(s')$ for every $s' \in B$ where $\mathcal R^{-1}(s') := \{s \in A: (s,s') \in \mathcal R  \}$.
\end{itemize}
Then, we have that
$$\mu(A) \leq \frac{k}{K}\mu(B)\,.$$
\end{lemma}
\begin{proof}
This is a simple application of ``counting in two ways'':
\begin{equation*}
\label{eq:mvmp1}
K\mu(A) \overset{(i)}\leq \sum_{(s, s') \in \mathcal R} \mu(s) = \sum_{s' \in B}\sum_{s \in \mathcal R^{-1}(s')}\mu(s)\overset{(ii)}\leq \sum_{s' \in B}k\mu(s') = k\mu(B)\,.  \qedhere
\end{equation*}
\end{proof}
Before we state the lemmata that are needed to invoke Lemma~\ref{lem:mvmp}, let us discuss the motivation behind them. Let us take the example of the proof of Lemma~\ref{lem:gluing3} (Lemma~\ref{lem:gluing2} is very similar).  Our goal is to apply the previous lemma for $A$ being the event under consideration in the lemma, and $B$ the full space of pairs of currents. Since a relation can be viewed as a multi-valued map, we can specify a relation on pairs of currents by describing several ways of modifying a given pair of currents $(\mb n_1, 
\mb n_2)$. In order to do so, we first perform the following two steps:
\begin{itemize}
	\item Select a family $Z$ of blocks ${\bf B}\in \mathcal B_n(\L_{N})$ that all intersect both $\mathcal C_{\mb n_1}(y)$ and $\mathcal C_{\mb n_1 + \mb n_2}(x)$. The number of choices for $Z$ will guarantee that $K$ is large.
	\item For each block ${\bf B}$, change the value of $\mb n_2$ along the edges of a shortest path $\Pi_{\bf B}$ between $\mathcal C_{\mb n_1}(y) \cap {\bf B}$ and $\mathcal C_{\mb n_1 + \mb n_2}(x) \cap{\bf B}$ so that these two sets become connected in ${\bf B}$ by the resulting pair of currents.
\end{itemize}
Let us call the new pair of currents $(\mb n_1', \mb n_2')$ (note that $\mb n_1' = \mb n_1$ at this stage). We want to be able to recover $Z$ from $(\mb n_1', \mb n_2')$ to ensure a small value of $k$ in Property~(ii) of Lemma~\ref{lem:mvmp} (called the \emph{reconstruction step} below). Since $x$ and $y$ are not connected in $\L_N$ by $\mb n_1 + \mb n_2$, a natural guess for $Z$ would include any block ${\bf B}$ containing a vertex $v \in \mathcal C_{\mb n_1}(y)$ which is connected to $x$ by $\n'_1+\n'_2$ through a path in 
$(\L_N\setminus \mathcal C_{\mb 
		n_1'}(y)) \cup \{v\}$.
Unfortunately, this guess may not be correct because changing $\mb n_2$ could have created many such 
vertices $v$ apart from the endpoints of the paths $\Pi_{\bf B}$. 

One way to address this problem is to additionally (and brutally) require that $(\mb n_1' + \mb n_2')(e) = 0$ for all the 
edges $e$ adjacent to the paths $\Pi_{\bf B}$ except those adjacent to their endpoints, but this has the disadvantage of 
introducing new sources to the currents. In order to remove these sources, we change the value of $(\mb n_1', \mb n_2')$ along some new paths connecting pairs of 
sources. These new paths should avoid the paths $\Pi_{\bf B}$ except, possibly, their endpoints in $\mathcal C_{\mb n_1}(y)$ and $\mathcal C_{\mb n_1 + \mb n_2}(x)$ so that we do not create any additional connections between $\mathcal C_{\mb 
	n_1}(y)$ and $x$. Furthermore, we want to find these paths in neighborhoods of fixed radius around the paths $\Pi_{\bf B}$ to prevent that the value of $(\mb n_1, \mb n_2)$ is changed on too many edges as that would, again, give a large value of $k$ in the 
reconstruction step. 

Our next lemma shows that it is always possible to find such paths within a graph distance of at most 2 from the paths $\Pi_{\bf B}$ 
for some specific choices of the latter. Due to purely technical reasons which we will leverage in our proof of Lemma~\ref{lem:gluing3}, we prove this lemma for $\mathcal C_{\mb n_1 + \mb n_2}(S)$ (with $S$ an arbitrary set) instead of $\mathcal 
C_{\mb n_1 + \mb n_2}(x)$. Without loss of generality, we implicitly assume in the rest of this paper that the size 
parameters $N$ and $n$ are integer powers of $2$.  Also we will use \emph{distance} for the $\ell_\infty$ distance and \emph{graph distance} for the $\ell_1$ distance on 
$\Z^d$. We say that $A$ is connected in $B$ if any two vertices of $A$ can be connected by a path of vertices in $B$. We refer the reader to Figure \ref{fig:lem:gluing} for a pictorial description of Lemma~\ref{lem:gluing}. 
\begin{lemma}
\label{lem:gluing}
Let $S \subset \L_N$ be such that $\mathcal C_{\mb n_1}(y) \cap \mathcal C_{\mb 
n_1 + \mb n_2}(S) = \emptyset$. Assume that there exists a block ${\bf B}\in\mathcal B_n(\Lambda_N)$ intersecting $\mathcal C_{\mb n_1 + \mb n_2}(S)$  and containing a  vertex  
in $\mathcal C_{\mb n_1}(y)$ with at least two neighbors in $\mathcal C_{\mb n_1}(y)$. 
Then, there exists a path $\Pi_{\bf B}=\Pi_{\bf B} (\n_1,\n_2)=(v_0, v_1, \ldots, v_{k})$ satisfying that
\begin{itemize}
\item
$v_0\in \mathcal C_{\mb n_1 + \mb n_2}(S)$ is within a distance of at most $3dn$ of the center of ${\bf B}$,   
\item $v_{k} \in \mathcal C_{\mb n_1}(y) $ is within a distance of at most $3dn$ of the center of ${\bf B}$, 
\item  $\Pi_{\bf B}$ is a shortest path between $v_0$ and $v_k$,
\item For every $0< i <k$, $v_i \notin \mathcal C_{\mb n_1}(y) \cup \mathcal C_{\mb n_1 + \mb n_2}(S)$,
\item The set
$T_{\bf B} = T_{\bf B}(\mb n_1, \mb n_2)$ of vertices in $\Lambda_N\setminus \mathcal C_{\mb n_1 + \mb n_2}(S)$ at a graph distance exactly 1 of $\Pi_{\bf B}\setminus\{v_0\}$ is connected in
the set $S_{\bf B}=S_{\bf B}(\mb n_1, \mb n_2)$ of vertices of $\Lambda_N\setminus \mathcal C_{\mb n_1 + \mb n_2}(S)$  which are either equal to $v_k$ or at a graph distance 1 or 2 of $\Pi_{\bf B}$.
\end{itemize}
\end{lemma}
 \begin{figure}[!htb]
   \centering
    \includegraphics[width=0.80\linewidth]{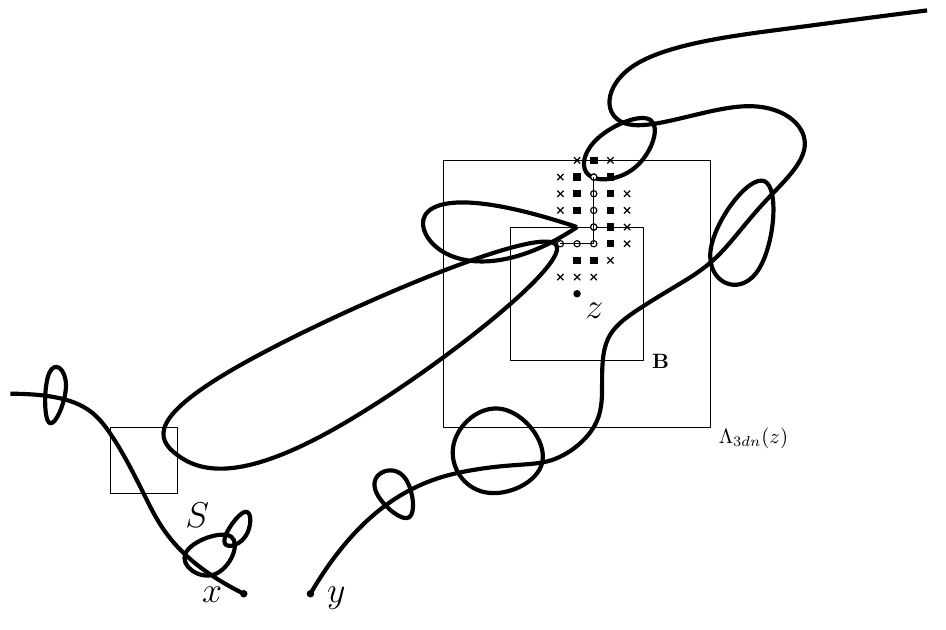}
    \caption{{\bf A schematic version of Lemma~\ref{lem:gluing}.} The thick lines represent the edges in the multigraph underlying $\mb n_1 + \mb n_2$ whereas the thin line represents the path $\Pi_{\bf B}=(v_0,\dots,v_k)$. We also depicted the sets $T_{\bf B}$ (filled squares) and $S_{\bf B} \setminus T_{\bf B} \cup \{v_k\}$ (crosses).}
    \label{fig:lem:gluing}
  \end{figure}
 We will prove Lemma~\ref{lem:gluing} (which is fairly technical but quite clear conceptually) at the end of the paper. With this lemma at our disposal, we can now formally describe the modification of $(\mb n_1, \mb n_2)$ that we discussed earlier. This is the content of the following lemma. 
 
 {\begin{remark}
 The requirement that the block ${\bf B}$ contains a vertex of $\mathcal C_{\n_1}(y)$ having two neighbors in $\mathcal C_{\n_1}(y)$ is used in this lemma and in this lemma only.
 \end{remark}}
 We introduce a few notation. From now on, we say that two blocks ${\bf B}$ and ${\bf B'}$ in $\mathcal B_n(\Lambda_N)$ are {\rm strongly disjoint} if their centers are at a distance of $7dn$ of each other. 
 Also, let $N_r(S)$ be the set of vertices of $\Z^d$ within a graph distance at most $r$ of $S$, and $E_r(S)$ be the set of edges with both endpoints in $N_r(S)$. 

 \begin{lemma}
	\label{lem:reconstruct}
	Assume that $\partial \mb n_1 = \{x, y\}$ and $\partial \mb n_2 = \emptyset$.  
	Let $ S \subset \Lambda_N$ be such that  $\mathcal C_{\mb n_1 + \mb n_2} (x) \cap S \neq \emptyset$ and  ${\mathcal C_{\mb n_1 + \mb n_2} (y)} \cap S = \emptyset$.
	Let $Z$ be a family of strongly disjoint blocks of $\mathcal B_n(\L_{N})$ with the property that every ${\bf B}\in Z$ intersects $\mathcal C_{\mb n_1 + \mb n_2}(S)$ and contains a vertex of $\mathcal C_{\mb n_1}(y)$ having at least two neighbors in $\mathcal C_{\mb n_1}(y)$.

	Then, there exists a new pair of currents $(\mb n_{1}', \mb n_{2}')$ (which is a function of $\n_1$, $\n_2$ and $Z$) with the following properties. If the paths $\Pi_{\bf B} = \Pi_{\bf B}(\mb n_1, \mb n_2) = (v_0^{\bf B}, v_1^{\bf B}, \ldots, v_{k_{\bf B}}^{\bf B})$ are given by Lemma~\ref{lem:gluing} above, we have that
	\begin{enumerate}[label = (\alph*)]
		\item
		$\partial \mb n_1' = \partial \mb n_1$ and $\partial \mb n_2' = \partial \mb n_2$.
		\item
		$(\mb n_1, \mb n_2)(e)\ne(\mb n_1', \mb n_2')(e) $ implies that $e\in E_2(\Pi_{\bf B})$ for some ${\bf B}\in Z$.
		\item $(\mb n_1, \mb n_2)(e)\ne(\mb n_1', \mb n_2')(e) $ implies that $\mb n_j'(e) \leq \mb n_j(e)$ or  $\mb n_j'(e) \leq 2$ for $j=1,2$.
		\item 
		The set of vertices $v \in  \mathcal C_{\mb n_{1}'}(y)$ that are connected by $\n'_1+\n'_2$ to $S$  in $(\L_N \setminus \mathcal C_{\mb n_{1}'(y)}) \cup \{v\}$ is exactly equal to the set of endpoints $v_{k_{\bf B}}^{\bf B}$ of the paths $\Pi_{\bf B}$ with ${\bf B}\in Z$.
	\end{enumerate}
	\end{lemma}
We already explained why Properties~(b) and (d) are important for the reconstruction step. Property~(c) is motivated by the observation that
$$\sum_{\mb n: \mb n \geq \mb n', \mb n_{|\mathsf E} = \mb n'_{|\mathsf E}}\frac{w(\mb n)}{w(\mb n')} \leq \e^{C|\mathsf E|}$$
for some $C > 0$ and any $\mathsf E \subset E_N$, and thus is crucial for efficient reconstruction. We will prove Lemma~\ref{lem:reconstruct} slightly later, but before that let us show how can it can be used to derive Lemmata~\ref{lem:gluing3} and \ref{lem:gluing2}. 
\begin{proof}[Proof of Lemma~\ref{lem:gluing3}]
The main idea underlying the proof is the following. We already know from Proposition~\ref{prop:connectivity} that ${\bf B}$ and ${\bf B}'$ are connected in $\mb n_1 + \mb n_2$ for every ${\bf B},{\bf B'}\in \mathcal B_n(\L_N)$ with an extremely high 
probability. As a consequence, we can effectively assume that this happens.  Then, it is easy to see that when $\mathcal C$ occurs and $x, y$ are not connected in $\L_N$ by $\mb n_1 + \mb n_2$, one can find a set $S$ and at least $N/(4n)$ many blocks ${\bf B} \in \mathcal B_n(\L_N)$ satisfying the conditions of Lemma~\ref{lem:gluing} (which are the same as the conditions on the 
elements of $Z$ in Lemma~\ref{lem:reconstruct}). Thus, we can pick any subset $Z$ of a certain number of such blocks with the additional condition that they are strongly disjoint, and modify $(\mb n_1, \mb n_2)$ according to the previous lemma. This gives us a relation on the set of pairs of currents and allows us to apply Lemma~\ref{lem:mvmp} for bounding the probability of the event we are interested in. As we will see later in the proof, the bound we obtain involves the 
(fixed) size of $Z$ as a parameter and the lemma follows by choosing an appropriate value for this size. The detailed argument is given below.

\bigbreak\noindent
{\bf Construction of $\mathcal R$.} Let $V(\mb n_1, \mb n_2)$ be a maximal subset of strongly disjoint blocks ${\bf B}\in \mathcal B_n(\L_N)$ 
with the property that $\mathcal C_{\mb n_1 + \mb n_2}(x) \cap {\bf B} =
	\emptyset$ and ${\bf B}$ contains a vertex of $\mathcal C_{\mb n_1}(y)$ having at least two neighbors in $\mathcal C_{\mb n_1}(y)$.
Let $\mathcal C_m$ denote the sub-event of $\mathcal C$ for which $V(\mb n_1, \mb n_2)$ contains exactly $m$ blocks.
	Also, let $\mathcal E$ be the event that every two blocks ${\bf B}$ and ${\bf B}'$ in $\mathcal B_n(\L_N)$ are connected by $\n_1+\n_2$ in $\L_N$. Finally, define
	$$\mathcal D_m := \mathcal C_m \cap \{\nlr{\L_N}{\mb n_1 + \mb n_2}{x}{y}\} \cap \mathcal E.$$
	\begin{figure}[!htb]
   \centering
    \includegraphics[width=.60\linewidth]{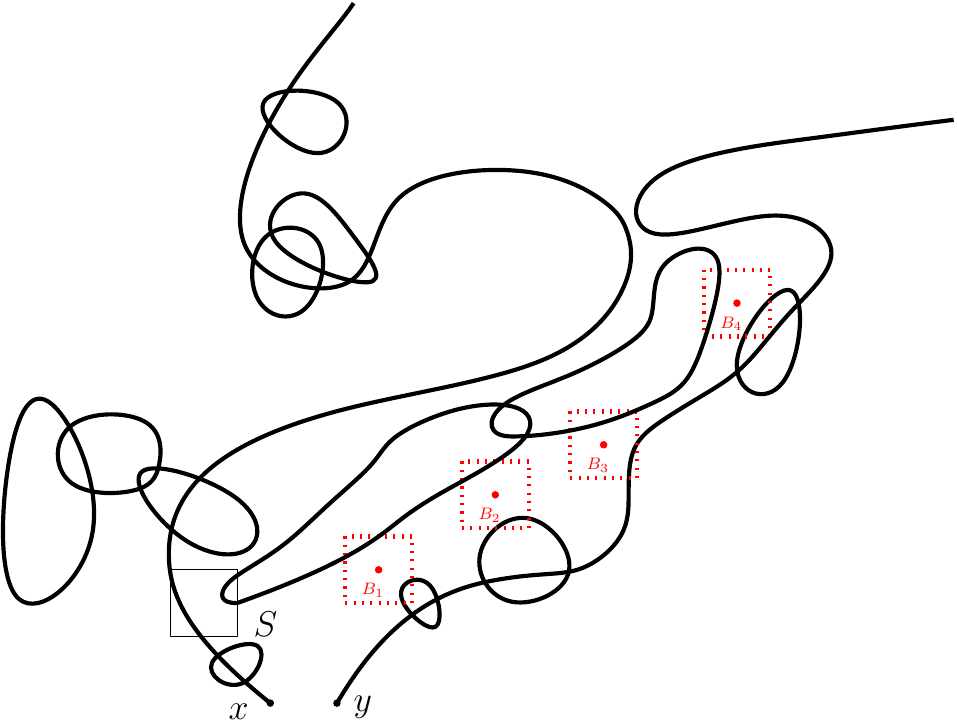}
    \caption{{\bf A simple representation of the set $V(\mb n_1, \mb n_2)$.}}
    \label{fig:lem:gluing2}
  \end{figure}
	Let $S = S(\mb n_1, \mb n_2) \in\mathcal B_n(\L_N)$ be a block intersecting $\mathcal C_{\mb n_1}(x)$ but not $\mathcal C_{\mb n_1 + \mb n_2}(y)$.  We remark that such a set $S$ exists since $\mathcal D_m \subset \mathcal C\subset\mathcal C_{y,x}$. 
	
The conditions of Lemma~\ref{lem:reconstruct} are met for $S$ and any subset $Z$ of 
$V(\mb n_1, \mb n_2)$.
Therefore, we can define a relation as 
follows. Fix $\delta>0$ to be a small number to be determined later. Then, $((\mb n_1, \mb n_2),(\mb n_1', \mb n_2'))\in \mathcal R$ if and only if $(\n_1,\n_2)\in\mathcal D_m$ and
$$\text{$(\n_1' ,\n'_2)= (\mb n_{1}',\n'_2)(\mb n_1, \mb n_2, Z)$ for some $Z \subset V(\mb n_1, \mb n_2)$ such that $|Z| = \delta m$},$$ 
where the map $(\n_1,\n_2,Z)\mapsto (\n'_1,\n_2')$ is given by Lemma~\ref{lem:reconstruct}.

With this definition of $\mathcal R$, we can now check Properties (i) and (ii) of Lemma~\ref{lem:mvmp} to deduce a bound on the probability of $\mathcal D_m$.
\bigbreak\noindent
{\bf Property (i).} For any $(\mb n_1, \mb n_2)\in\mathcal D_m$, the map $Z \mapsto \{v_{k_{\bf B}}^{\bf B}: {\bf B} \in Z\}$ is one-to-one since $v^{\bf B}_{k_{\bf B}}$ is within a distance $3dn$ of the center of ${\bf B}$, and the blocks in $V(\mb n_1, \mb n_2)$ are strongly disjoint. 
Hence,
\begin{equation}
\label{eq:gluing3_Rn12card}
|\mathcal R((\mb n_1, \mb n_2))| \geq \binom{|V(\mb n_1, \mb n_2)|}{\delta m} =\binom{m}{\delta m} \,.
\end{equation}
\bigbreak\noindent
{\bf Property (ii).} Fix $(\mb m_1, \mb m_2)$. For a block $S$, a set of $\delta m$ strongly disjoint blocks $Z$, a subset of edges $E$ and a collection of paths $\Pi=(\Pi_{\bf B}:{\bf B}\in Z)$, introduce the set $\mathcal A(\mb m_1,\mb m_2,S,Z,E,\Pi)$ of pairs of currents $(\n_1,\n_2)\in \mathcal D_m$ such that 
\begin{itemize}[noitemsep,nolistsep]
\item $S(\n_1,\n_2)=S$,
\item $(\mb m_1,\mb m_2) =(\mb n_1',\n'_2) (\n_1,\n_2,Z)$,
\item $E=\{e\in E_N:(\n_1,\n_2)(e)\ne (\mb m_1,\mb m_2)(e)\}$,
\item $\Pi_{\bf B}(\n_1,\n_2)=\Pi_{\bf B}$ for every ${\bf B}\in Z$. 
\end{itemize}
The definition of $\mathcal R$ implies directly that
\begin{align}
\label{eq:gluing3_reconstruct1}
\sum_{(\mb n_1, \mb n_2) \in \mathcal R^{-1}(\mb m_1, \mb m_2)}&\P_{\L_N^+, \L_N}^{\{x, y\}, \emptyset}[(\mb n_1, \mb n_2)] \nonumber \\
&= \P_{\L_N^+, \L_N}^{\{x, y\}, \emptyset}[(\mb m_1, \mb m_2)]\sum_{(S, Z, E, \Pi)}\sum_{(\mb n_1, \mb n_2) \in \mathcal A(\mb m_1,\mb m_2,S,Z,E,\Pi)} \frac{w(\mb n_1)w(\mb n_2)}{w(\mb m_1)w(\mb m_2)}.
\end{align}
We will bound the two summations on the right of \eqref{eq:gluing3_reconstruct1} in two steps. 
First, Properties (b) and (c) of Lemma~\ref{lem:reconstruct} imply that 
\begin{align}
\sum_{(\mb n_1, \mb n_2) \in \mathcal A(\mb m_1,\mb m_2,S,Z,E,\Pi)} \frac{w(\mb n_1)w(\mb n_2)}{w(\mb m_1)w(\mb m_2)}&\leq \prod_{\substack{e \in E\\ j=1,2}}\Big(\sum_{\ell \geq \mb m_j(e)}\frac{\beta^{\ell}\mb m_j(e)!}{\beta^{\mb m_j(e)}\ell!} + \sum_{0 \leq \ell \leq \mb m_j(e)}\frac{\beta^{\ell}\mb m_j(e)!}{\beta^{\mb m_j(e)}\ell!}\mathbb I[\mb m_j(e) \leq 2]\Big)\nonumber\\
& \leq   \exp ({C|E|})\le \exp(C\delta mn).\label{eq:bound ratio}
\end{align}
In the second line, we used that $E$ is included in the union of the $E_2(\Pi_{\bf B})$ for ${\bf B}\in Z$. Furthermore, $E_2(\Pi_{\bf B})$ is the set of vertices at a graph distance at most 2 from $\Pi_{\bf B}$.
According to Lemma~\ref{lem:gluing}, $\Pi_{\bf B}$ must be a shortest length path between two vertices within a distance $3dn$ of the center of ${\bf B}$, a fact which implies that its length is smaller than $6d^2n$. Overall, this implies that $|E|\le 6d^2n|Z|\le C\delta mn$.

Second, we bound the number of possibilities for $S$, $Z$, $E$ and $\Pi$. Obviously, there are fewer than $|\mathcal B_n(\L_N)|$ choices for $S$. Property (d) of Lemma~\ref{lem:reconstruct} implies that $(\mb m_1,\mb m_2)$ and $S$ determine the points $v_{k_{\bf B}}^{\bf B}$ for ${\bf B}\in Z$. Since these vertices are within a distance $3dn$ of the centers of the blocks in $Z$, and that there are $\delta m$ blocks in $Z$, this 
reduces the number of possibilities for $Z$ to $(2d)^{3d\delta m}$. Also, each one of the paths $\Pi_{\bf B}$ is a self-avoiding path of length at most $6d^2n$ ending at $v_{k_{\bf B}}^{\bf B}$, and therefore there are at most $(2d)^{6d^2n\delta m}$ choices for the collection of paths $\Pi_{\bf B}$ with ${\bf B}\in Z$. Finally, $E$ being a subset of $\cup_{{\bf B} \in Z}E_2(\Pi_{\bf B})$, we deduce that the number of possibilities for $E$ is bounded by $2^{25d^3n\delta m}$.

Overall plugging these bounds and \eqref{eq:bound ratio} into \eqref{eq:gluing3_reconstruct1} gives
\begin{equation*}
\label{eq:gluing3_reconstruct}
\sum_{(\mb n_1, \mb n_2) \in \mathcal R^{-1}(\mb m_1, \mb m_2)}\P_{\L_N^+, \L_N}^{\{x, y\}, \emptyset}[(\mb n_1, \mb n_2)] \leq N^d\e^{C\delta mn}\,\P_{\L_N^+, \L_N}^{\{x, y\}, \emptyset}[(\mb m_1, \mb m_2)]\,.
\end{equation*}
\bigbreak\noindent
{\bf Conclusion of the proof.} Plugging the last inequality and \eqref{eq:gluing3_Rn12card} in Lemma~\ref{lem:mvmp} gives
\begin{equation}
\label{eq:gluing3_Bnarmsbnd}
\P^{\{x,y\}, \emptyset}_{\L_N^+, \L_N}[\mathcal D_m] \leq \frac{CN^d\e^{C\delta mn}}{\binom{m}{\delta m}}\,.
\end{equation}
At this stage, recall that $N=e^{n^\alpha}$ with $\alpha>1$. An elementary computation shows that
choosing $\delta = \e^{-2Cn}$ (which is a valid choice since $\delta m \geq \e^{-2Cn} N / (4n) = \e^{n^\alpha - 2Cn} / (4n) > 1$ for every large $n$) gives that 
\begin{equation*}
\label{eq:gluing3_lemma1}
\P^{\{x, y\}, \emptyset}_{\L_N^+, \L_N}[\mathcal D_m] \leq \e^{-c\e^{-Cn} mn}\,
\end{equation*}
for every $m \geq \tfrac1{4(2d)^{4d}}N/n$ and $N$ large enough (independent of $m$). Since any pair $(\n_1,\n_2)\in\mathcal C$ must be in one of the $\mathcal C_m$ for $m \geq \tfrac1{4(2d)^{4d}}N/n$, we obtain that \begin{equation}
\label{eq:gluing3_lemma2}
\P^{\{x, y\}, \emptyset}_{\L_N^+, \L_N}[\mathcal C \cap \{\nlr{\L_N}{\mb n_1 + \mb n_2}{x}{y}\} \cap \mathcal E] \leq \e^{-c\e^{-Cn}N}
\end{equation}
by summing over $m$. Now, since $N = \e^{n^\alpha}$ for $\alpha < d-1$, it follows from Proposition~\ref{prop:connectivity} that
	\begin{equation}
	\label{eq:gluing3_probElb}
	\P_{\L_N^+, \L_N}^{\{x, y\}, \emptyset}[\mathcal E] \geq 1 - CN^d\e^{-cn^{d-1}} \geq 1 - \e^{-cn^{d-1}}\,
	\end{equation}
	for every large enough $N$. 
The result follows readily from this bound and \eqref{eq:gluing3_lemma2}.
\end{proof}
{\begin{remark}\label{rmk:h}
We already used the bound $\alpha<d-1$ to deduce Proposition~\ref{thm:mixing1} from Lemmata~\ref{lem:gluing3} and \ref{lem:gluing2}. The end of the previous proof further explains where $1<\alpha<d-1$ is used. Indeed, $\alpha<d-1$ is used to invoke Proposition~\ref{prop:connectivity}. The bound $\alpha>1$ is there to guarantee that the exponential (in $n$) finite-energy cost appearing for instance in \eqref{eq:bound ratio} is overcome by the choice of $N$.
\end{remark}}
\begin{proof}[Proof of Lemma~\ref{lem:gluing2}] The proof is the same as the previous one, with $S=\{x\}$ instead of $S$ being a block.
\end{proof}
We now focus on the proof of Lemma~\ref{lem:reconstruct} using Lemma~\ref{lem:gluing}.
\begin{proof} [Proof of Lemma~\ref{lem:reconstruct}] We begin with the construction of $(\n_1',\n'_2)$ and then show that Properties~(a)--(d) are satisfied. 
\bigbreak\noindent
{\bf Construction of $(\n_{1}', \n_{2}')$.} We proceed in three steps roughly corresponding to the steps that we outlined before Lemma~\ref{lem:gluing}, albeit in a slightly 
	different order. For clarity of exposition, we will use different notations for the pairs of 
	currents resulting from different steps. As a result, we construct three intermediate currents called $\n_1^0$, $\n_2^0$, and $\n_2^1$. At the conclusion of each step, we  track the sources of the currents constructed in that step as well as the direction in which their value changes from the previous step. This will help us to verify Properties~(a) and (c) at the end. 
	
	Also, since the construction will obviously be independent for every ${\bf B}\in Z$ (since the blocks are strongly disjoint), we present it only near one prescribed block ${\bf B}$, which we remove from the notation for convenience. We therefore write $\Pi=(v_0,\dots,v_k)$. Also, we set $\overline{\Pi}=(v_1,\dots,v_{k-1})$ to be the path $\Pi$ minus its endpoints. We will return to the original notation later when we verify the properties of $(\mb n_1', \mb n_2')$.
\bigbreak
\noindent \emph{Step 1: Closing the edges adjacent to $\Pi$.}
	We want to make both currents 0 on edges adjacent to $\Pi$ except those adjacent to its endpoints. Formally, set	\begin{align}
	\label{eq:gluing_n10def}
	\mb n_{1}^0 (e) &:= 
	\begin{cases}
	0  & \mbox{if }e = \{v, w\} \mbox{ for $v\in \overline\Pi$ and $w\in \L_N$}, \\
	\mb n_1(e) & \text{otherwise},
	\end{cases}\\
	\label{eq:gluing_n20def}
	\mb n_{2}^0 (e) &:=
	\begin{cases}
	0 & \mbox{if }e = \{v,w\} \mbox{ for $v\in \overline\Pi$ and $w\in \L_N\setminus \Pi$}, \\
	\mb n_2(e) & \text{otherwise}\,.
	\end{cases}
	\end{align}
In order to track the sources of $\mb n_1^0$, let us first notice that vertices in $\overline\Pi$ are not sources since they are incident to edges with zero $\n_1^0$ current only.
	Since $\overline \Pi$ does not intersect $\mathcal C_{\mb n_1}(y)$ or $\mathcal C_{\mb n_1 + \mb n_2}(S)$ (by the definition of $\Pi$ in Lemma \ref{lem:gluing}), we deduce that the sources of $\n_1^0$ not equal to $x$ or $y$ and within a distance of $4dn$ of the center of ${\bf B}$ must be at a distance exactly equal to 1 of $\overline \Pi$ and that at least one of the edges incident to them must have changed value between $\n_1$ and $\n_1^0$. Since edges between $\overline\Pi$ and $\mathcal C_{\mb n_1 + \mb n_2}(S) \cup \mathcal C_{\mb n_1}(y) $ already had a zero current in $\n_1$, we deduce that the new sources cannot be in $\mathcal C_{\n_1+\n_2}(S)$ and must therefore be in the set $T_{\bf B} = T_{\bf B}(\mb n_1, \mb n_2)$ defined in Lemma~\ref{lem:gluing}.
Also, notice that there is an even number of such sources.  
Similarly, we get that the set of sources of $\n_2^0$ is a subset of $\overline\Pi\cup T_{\bf B}$.
	
	Finally, notice that 
	\begin{equation}
	\label{eq:gluing_njS0njineq}
	\mb n_{1}^0(e) \leq \mb n_1(e) \mbox{ and } \mb n_{2}^0(e) \leq \mb n_2(e) \mbox{ for every } e \in E_N\,.
	\end{equation} 
	\bigbreak
	\noindent{\em Step~2. Opening the edges along $\Pi$.} The second step consists in defining $\n_1^1=\n_1^0$ and 
	\begin{equation*}
	\label{eq:gluing_n21def}
	\mb n_{2}^1 (e): =
	\begin{cases}
	2  & \mbox{if } e = \{v, w\} \mbox{ with } v,w\in\Pi \mbox{ or if }v = v_k \mbox{ and }w \in \L_N.\\
	\mb n_{2}^0(e) & \text{otherwise}.
	\end{cases}
	\end{equation*} 
By definition, 
\begin{equation}
\label{eq:gluing_nj1ineq}
\mbox{$\mb n_{2}^1(e) \neq \mb n_{2}^0(e)$ implies that 
	$\mb n_{2}^1(e) = 2$}.
\end{equation}
	Since $\overline \Pi$ does not intersect $\mathcal C_{\mb n_1 + \mb n_2}(S)$, we have $\mb n_2(\{v_{k-1}, v_k\}) = \mb n_2^0(\{v_{k-1}, v_k\}) = 0$ and therefore the definition of $\mb n_2^1$ also implies that a source of $\mb n_2^1$ is a source of $\n_2^0$ which is not on $\Pi$. Hence, it is included in $T_{\bf B}$ again and also, there are an even number of sources within a distance of $4dn$ of ${\bf B}$. 
	\bigbreak	\noindent{\em Step~3. Killing the sources of $(\mb n_1^1, \mb n_2^1)$.}
We now remove the additional sources of $\mb n_1^1$ and $\mb n_2^1$, all of which 
	lie in $T_{\bf B}$. We start with the sources of $\n_1^1$. By the fifth item of Lemma~\ref{lem:gluing}, $T_{\bf B}$ is a connected subset of $S_{\bf B}$. We can therefore choose (in some arbitrary but fixed manner) paths $\Gamma_{1}^1, \ldots, \Gamma_{\ell}^1$ in $S_{\bf B}$ pairing the sources of $\n_1^1$. Let $\mb m$ be the current, equal at each edge to 0 (resp.~1) if there is an even (resp.~odd) number of paths going through this edge. Finally, set	\begin{equation*}
	\label{eq:gluing_njdef}
	\mb n_1'(e) :=
	\begin{cases}
	\mb n_1^1(e) - \mb m(e)  & \mbox{if }\mb n_1^1(e) \ge2\\
	\mb n_1^1(e) + \mb m(e) & \text{otherwise}\,.
	\end{cases}
	\end{equation*}
We obtain immediately that $\partial\mb n_1'=\{x,y\}$. We proceed in the same way for $\n_2^1$ in order to obtain $\n_2'$ with $\partial\n_2'=\emptyset$.	Again, notice that
\begin{equation}
\label{eq:gluing_nj'ineq}
\mb n_{j}'(e) \leq \mb n_j^1(e) \mbox{ or } \mb n_{j}'(e) \leq 2 \mbox{ for every } e \in E_N \mbox{ and }j = 1,2\,.
\end{equation} 	
	\bigbreak\noindent
	{\bf Verification of the properties of $(\mb n_1', \mb n_2')$.} In this part, we assume that we made the construction above for every ${\bf B}\in Z$. 
	\bigbreak\noindent
{\em Property (a).} This follows readily from the two sentences preceding \eqref{eq:gluing_nj'ineq}.
\bigbreak\noindent
{\em Property (b).} At each stage of the construction, edges for which the value of currents is modified are always within a graph distance $2$ from one of the paths $\Pi_{\bf B}$, which means that Property (b) is satisfied. 
\bigbreak\noindent
{\em Property (c).} This follows from the displays \eqref{eq:gluing_njS0njineq}, \eqref{eq:gluing_nj1ineq} and \eqref{eq:gluing_nj'ineq}.
 \bigbreak\noindent
{\em Property (d).} Let us first verify one direction of Property~(d), namely that for any $\mb B \in Z$, $v_{k_{\bf B}}^{\bf B} \in \mathcal C_{\mb n_1'}(y)$ and it is connected in $(\L_N \setminus 
\mathcal C_{\mb n_{1}'(y)}) \cup \{v_{k_{\bf B}}^{\bf B}\}$ to $S$ by $\mb n_1' + \mb n_2'$. We divide the proof into two steps.
\medbreak{\em Proof that $v_{k_{\bf B}}^{\bf B} \in \mathcal C_{\mb n_1'}(y)$.} It suffices to show that $\mathcal C_{\mb n_1}(y) \subset \mathcal C_{\mb n_1'}(y)$. It is clear from our construction in Step~3 that $\mathcal C_{\mb n_1^0}(y) = \mathcal C_{\mb 
n_1^1}(y) \subset \mathcal C_{\mb n_1'}(y)$. Also, the definition of $\mb n_1^0$ implies that $\mb n_1^0(e) < \mb n_1(e)$ only if $e = \{v, w\}$ for $v \in \overline \Pi_{{\bf B}}$ where $\mb  B \in Z$ and $w \in \L_N$. Since $\overline \Pi_{\bf B}$ does not intersect $\mathcal C_{\mb n_1}(y)$ by Lemma~\ref{lem:gluing}, we therefore get $\mathcal C_{\mb n_1^0}(y) = \mathcal C_{\mb n_1}(y)$ which completes the proof.

\medbreak{\em Proof that $v_{k_{\bf B}}^{\bf B}$ is connected in $(\L_N \setminus 
	\mathcal C_{\mb n_{1}'(y)}) \cup \{v_{k_{\bf B}}^{\bf B}\}$ to $S$ by $\mb n_1' + \mb n_2'$.} Notice that the path $\Pi_{\bf B}$ is open in $\mb n_2'$ and connects $v_{k_{\bf B}}^{\bf B}$ to $v_0^{\bf B}$. Since $v_0^{\bf B} \in \mathcal C_{\mb n_1 + \mb n_2}(S)$ by definition, it therefore suffices to show that $\mathcal C_{\mb n_1 + \mb n_2}(S) \subset \mathcal C_{\mb n_1' + \mb n_2'}(S)$ and that $\overline \Pi_{\bf B}$ and $\mathcal C_{\mb n_1 + \mb n_2}(S)$ do not intersect $\mathcal C_{\mb n_1'}(y)$. 
	
	To this end, let us first recall from the definitions of $(\mb n_{1}^1, \mb n_{2}^1)$ and $(\mb n_{1}', \mb n_{2}')$ that $(\mb n_1' + \mb n_2')(e) > 0$ whenever $(\mb n_1^0 + \mb n_2^0)(e) > 0$, and consequently
	 $\mathcal C_{\mb n_{1}^0 + \mb n_{2}^0}(S)\subset \mathcal C_{\mb n_{1}' + \mb n_{2}'}(S)$. Also, from the construction in Step~1 we have $(\mb n_1^0 + \mb n_2^0)(e) < (\mb n_1 + \mb n_2)(e)$ only if $e$ has an endpoint in $\overline \Pi_{\mb B}$ for some $\mb B \in Z$. Since $\overline\Pi_{\bf B}$ does not intersect $\mathcal C_{\mb n_1 + \mb n_2}(S)$ by Lemma~\ref{lem:gluing}, it therefore follows that $\mathcal C_{\mb n_{1}^0 + \mb n_{2}^0}(S) = \mathcal C_{\mb n_{1} + \mb n_{2}}(S)$ and hence $\mathcal C_{\mb n_1 + \mb n_2}(S) \subset \mathcal C_{\mb n_1' + \mb n_2'}(S)$. 
	 
	 Next let us ``bound'' the set $\mathcal C_{\mb n_1'}(y)$ from above. Notice that
	\begin{align*}
	\mathcal C_{\mb n_{1}'}(y) &\subset \mathcal C_{\mb n_{1}^0}(y) \cup \bigcup_{{\bf B} \in Z}\mathcal C_{\mb n_{1}^0}(S_{\bf B}) = \mathcal C_{\mb n_1}(y) \cup \bigcup_{{\bf B} \in Z}\mathcal C_{\mb n_{1}^0}(S_{\bf B})\,,
	\end{align*}
	where in the second step we used $\mathcal C_{\mb n_1}(y) = \mathcal C_{\mb n_1^0}(y)$ as proved in the previous part. Since, by Lemma~\ref{lem:gluing}, $\overline\Pi_{\mb B}$ is disjoint from $\mathcal C_{\mb n_1}(y)$ and $S_{\mb B}$, we deduce from the previous displayed equation that $\mathcal C_{\mb n_{1}'}(y) \cap \overline \Pi_{\mb B} = \emptyset$ for any $\mb B \in Z$ if we show $\mathcal C_{\mb n_1^0}(\overline\Pi_{\mb B})= \overline\Pi_{\mb B}$ for any such $\mb B$. But this follows immediately from the fact that $\mb n_{1}^0(e) = 0$ for any edge $e$ whose one endpoint lies in $\overline\Pi_{\bf B}$ for some ${\bf B}\in Z$. 
	
	Similarly, from $\mathcal C_{\mb n_1}(y) \cap \mathcal C_{\mb n_1 + \mb n_2}(S) = \emptyset$ and $S_{\mb B} \cap \mathcal C_{\mb n_1 + \mb n_2}(S) = S_{\mb B} \cap \mathcal C_{\mb n_1^0 + \mb n_2^0}(S) = \emptyset$ (both are consequences of Lemma~\ref{lem:gluing}), we deduce that $\mathcal C_{\mb n_1'}(y) \cap \mathcal C_{\mb n_1 + \mb n_2}(S) = \emptyset$, concluding the proof of this part.

	\bigbreak It remains to verify the other direction of Property~(d), namely that any vertex $v \in  \mathcal C_{\mb n_{1}'}(y)$ that is connected to $S$ by $\n'_1+\n'_2$ in $(\L_N \setminus \mathcal C_{\mb n_{1}'(y)}) \cup \{v\}$ must be the endpoint $v_{k_{\bf B}}^{\bf B}$ of a path $\Pi_{\bf B}$ for some ${\bf B}\in Z$. It suffices to show that any self-avoiding path $\Pi$ in $\mb n_{1}' + \mb n_{2}'$ between $S$ and $\mathcal C_{\mb n_{1}'}(y)$ 
	contains $v_{k_{\bf B}}^{\bf B}$ for some ${\bf B}\in Z$. 
	
	Since the blocks are strongly disjoint, we deduce from the definition of $(\mb n_{1}', \mb n_{2}')$ that 
	$$(\mb n_{1}' + \mb n_{2}')(e) = (\mb n_{1}^1 + \mb n_{2}^1)(e) = 0\,$$
	for every edge with one endpoint in some $\overline\Pi_{\bf B}$. Together with the fact that $\overline \Pi_{\bf B}$ is disjoint from $\mathcal C_{\mb n_1'}(y)$ and $S$ (already noted in the previous parts), this implies that if $\Pi$ contains an edge of $\Pi_{\bf B}$, then it must contain the vertex $v_{k_{\bf B}}^{\bf B}$. On the other hand, if $\Pi$ does not contain any edge of $\Pi_{\bf B}$ for any ${\bf B} \in Z$, then it must contain an edge with one endpoint in $\mathcal C_{\mb n_{1} + \mb n_{2}}(S)$ and another which is not in $\mathcal C_{\mb n_{1} + 
		\mb n_{2}}(S)$ or in any of the $\overline\Pi_{\bf B}$. This is because $\mathcal C_{\mb n_1'}(y) \cap \mathcal C_{\mb n_1 + \mb n_2}(S) = \emptyset$ as we noted in the previous part. 
		Obviously such an edge cannot exist in $\mb n_1 + \mb n_2$. Now, observing that $(\mb n_1 + \mb n_2)(e) < (\mb n_1' + \mb n_2')(e)$ only if $e$ is an edge in $\Pi_{\mb B}$ or $S_{\mb B}$ for some $\mb B \in Z$ and that $S_{\mb B} \cap \mathcal C_{\mb n_1 + \mb n_2}(S) = \emptyset$, we conclude that such an edge cannot exist in $\mb n_{1}' + \mb n_{2}'$ as well, thus finishing the proof.\end{proof}
Finally, we are left with the proof of Lemma~\ref{lem:gluing}. It is clear that the only non-trivial part of the lemma is the fifth item. However, the following crucial observation makes it much simpler. 
Suppose that we choose $\Pi_{\bf B}\coloneqq (v_0, \ldots, v_k)$ as a shortest path between $\mathcal C_{\mb n_1 + \mb n_2}(S)$ and $\mathcal C_{\mb n_1}(y)$ restricted to 
some region. Also suppose that $N_2(\Pi_{\mb B})$ lies in that region. Then the distance between $v_t$ and $\mathcal C_{\mb n_1 + \mb n_2}(S)$ is at least 3 for any $t \geq 3$ (since otherwise there would be a shorter path between $\mathcal C_{\mb n_1 + \mb n_2}(S)$ and $\mathcal C_{\mb 
n_1}(y)$ restricted to the region) and hence $N_2(v_t) \setminus (\Pi_{\bf B} \cup \mathcal C_{\mb n_1 + \mb n_2}(S)) = N_2(v_t) \setminus \Pi_{\bf B}$ for any such 
$t$. Now, in dimension $d \geq 3$, it is not difficult to prove that $N_1(v_t) \setminus \Pi_{\bf B}$ is a connected subset of $N_2(v_t) \setminus \Pi_{\bf B}$ if $\Pi_{\bf B}$ is a shortest path between its endpoints. Hence, any two vertices in $N_1(v_t) \setminus \Pi_{\bf B}$ can be connected using a path in $N_2(v_t) \setminus \Pi_{\bf B}$ that does not intersect $\mathcal C_{\mb n_1 + \mb 
n_2}(S)$. Therefore the only ``problematic'' vertices in $T_{\bf B}$ (see the statement of 
Lemma~\ref{lem:gluing}) are those in $N_1(v_1)$ and $N_1(v_2)$. We deal with them in the proof by carefully considering all possible alignments for the first three edges of $\Pi_{\bf B}$.
\begin{proof}[Proof of Lemma~\ref{lem:gluing}] For the purpose of this proof, we use $\mb e_j$ to denote the vertex in $\Z^d$ whose $j$-th coordinate is 1 
and all the other coordinates are 0. Let $u$ be a vertex in $\mathcal C_{\mb n_1}(y) \cap {\bf B}$ with at least 
two neighbors in $\mathcal C_{\mb n_1}(y)$. Since ${\bf B}$ intersects $\mathcal C_{\mb n_1 + \mb n_2}(S)$, there is a vertex $v \in \mathcal C_{\mb n_1 + \mb n_2}(S)$ which is within a distance of $3dn$ from the center of ${\bf B}$ realizing the graph distance, denoted by $d_1$ below, between $u$ and $\mathcal C_{\mb n_1 + \mb n_2}(S)$ restricted to the box of radius $4dn$ with the same center as $\mb B$.

Now consider a shortest (for the length) path $\Pi'\coloneqq (v = v_0', v_1', \ldots, v_{d_1}' = u)$ between $v$ and $u$ as our first choice for $\Pi_{\mb B}$. 
For any two vertices $p, q$ adjacent to $v_t$ which do not lie in $\Pi'$ or $\mathcal C_{\mb n_1 + \mb n_2}(S)$ (where $t \geq 1$), we would like to connect them by a path in $\{v'_{d_1}\} \cup N_2(\Pi') \setminus \Pi'$ which 
avoids $\mathcal C_{\mb n_1 + \mb n_2}(S)$. In what follows we do this based on the value of $t$.

\smallskip When $t \geq 3$, the observation we made before the proof implies that the graph distance between any vertex in $N_2(v_t')$ and $u$ is strictly less than $d_1$ and hence $N_2(v_t')$ does not intersect $\mathcal C_{\mb 
n_1 + \mb n_2}(S)$. Thus we only need to show that, for $t \geq 3$, the set $N_1(v_t') \setminus \Pi'$ is connected in $N_2(v_t') \setminus \Pi'$. To this end we consider two distinct possibilities for a pair of vertices $w, w'$ in $N_1(v_t') \setminus \Pi'$. The first possibility is that $w = v_t' + \mb e$ and $w' = v_t' + \mb e'$ for some $\mb e, \mb e' \in \{\pm e_j: j \leq d\}$ such that $\mb e 
\neq -\mb e'$. Notice that in this case, the vertex $v_t' + \mb e + \mb e'$ cannot lie in $\Pi'$ since otherwise the segment of $\Pi'$ between $v_t'$ and $v_t' + \mb e + \mb e'$ would contain at least 3 edges contradicting the 
fact that $\Pi'$ is a shortest path. Hence the path $(w, v_t' + \mb e + \mb e', w')$ lies in $N_2(v_t') \setminus 
\Pi'$. The second possibility is that $w = v_t' - \mb e$ and $w' = v_t' + \mb e$ for some $\mb e \in \{\pm \mb e_j: j \leq d\}$ which we can assume, without loss of generality, to be $\mb e_1$. Let $\mb e \in \{\pm \mb e_j: j = 2, 3\}$ and consider the path $(w, w + \mb e, v_t' + \mb e, w' + 
\mb e, w')$ in $N_2(v_t')$. If this path intersects $\Pi'$, then our previous argument yields that $v_t' 
+ \mb e \in \Pi'$. Since $\Pi'$ is a self-avoiding path, it cannot contain more than 2 neighbors of $v_t'$ and thus the path $(w, w + \mb e, v_t' + \mb e, w' + \mb e, w')$ lies in $N_2(v_t') \setminus \Pi'$ for some $\mb e \in \{\pm \mb e_j: j = 1, 2\}$.

\smallskip Thus it only remains to deal with the vertices adjacent to $v_1$ and $v_2$ which do not lie in $\Pi'$ or $\mathcal C_{\mb n_1 + \mb 
n_2}(S)$. Unfortunately it may not be always possible to connect a pair of such vertices by a path in $\{v_{d_1}'\}\cup N_2(\Pi') \setminus \Pi'$ and in those cases we need to modify $\Pi'$. Below we discuss these cases based on all possible values of $d_\infty \coloneqq 
\|v - u\|_\infty$ and $d_1$. One observation that will be particularly useful is that $N_1(v_2')$ and $\mathcal 
C_{\mb n_1  + \mb n_2}(S)$ are disjoint. Thus we only need to connect any vertex in $N_1(v_1') \setminus (\Pi' \cup \mathcal C_{\mb n_1 + \mb n_2}(S))$ to a vertex in $N_1(v_t') \setminus \Pi'$ for some $t \geq 2$ and similarly any vertex in $N_1(v_2') 
\setminus \Pi'$ to a vertex in $N_1(v_t') \setminus \Pi'$ for some $t \geq 3$ using a path in $N_2(\Pi') \setminus (\Pi' \cup \mathcal C_{\mb n_1 + \mb n_2}(S))$. 
As a final remark before we go to the details, let us mention that when $d_1 = 1$ or when $d_\infty = 2$ and $d_1 = 2$, it is not difficult to see that either $\Pi'$ or $(v_0', v_1')$ satisfies the 
items of the lemma. Hence, we only focus on the other cases.
\bigbreak\noindent
{\bf Case~1.} $d_\infty \geq 3$.
\medbreak
In this case we can choose $\Pi'$ so that $v_t' = v_0' + t  \mb e$ for every $t \le 3$ and some $\mb e \in \{\pm \mb e_j: j \le d\}$. Without loss of generality, we assume 
that $\mb e = \mb e_1$. 
Notice that for any vertex $p$ in $N_1(v_1') \setminus (\Pi' \cup \mathcal 
C_{\mb n_1 + \mb n_2}(S))$, the vertex $p + \mb e_1$ lies in $N_1(v_1' + \mb e_1) \setminus \Pi' = N_1(v_2')\setminus \Pi'$ and hence the edge $\{p, p + \mb e_1\}$ lies in $N_1(\Pi') \setminus (\Pi' \cup \mathcal 
C_{\mb n_1 + \mb n_2}(S))$ as well. 
Similarly any vertex in $N_1(v_2') \setminus 
\Pi'$ is either a neighbor of $v_4'$ (when $v_4' = v_3' \pm \mb e_j$ for some $j > 1$) or has a neighbor in $N_1(v_2' + \mb e_1) \setminus \Pi' = N_1(v_3') \setminus \Pi'$ (when $v_4' = v_3' + \mb e_1$). 
These together imply that the items of Lemma~\ref{lem:gluing} hold for $\Pi_{\mb B} = (v_0', v_1', \ldots, v_{t}')$, where $t = \min \{t' \le d_1: v_{t'}' \in \mathcal C_{\mb n_1}(y)\}$.
\bigbreak
\noindent {\bf Case~2.} $d_\infty = 2$ and $d_1 \geq 3$.
\medbreak
This is the most involved case. Here, we can assume without loss of generality that $v_t' = v_0' + t \mb e_1$ for $t \le 2$ and $v_3' = v_2' + 
\mb e_2$. Just like in Case~1, any vertex in $N_1(v_1') \setminus (\Pi' \cup \mathcal C_{\mb n_1 + \mb n_2}(S))$ has a neighbor in $N_1(v_2') \setminus (\Pi' \cup \mathcal 
C_{\mb n_1 + \mb n_2}(S))$. Also notice that for any vertex $p$ in $N_1(v_2') \setminus (\Pi' \cup \mathcal C_{\mb n_1 + \mb n_2}(S))$ other than $v_2' - \mb e_2$, the vertex $p + \mb e_2$ lies in $N_1(v_3') \setminus 
\Pi'$. Hence we only need a separate treatment for $v_2' - \mb e_2$, i.e. when it does not lie in $\mathcal C_{\mb n_1 + \mb n_2}(S)$. To this end we consider several sub-cases based on the neighboring vertices of $v_2' - \mb e_2$ and the value of $d_1$.

\medskip\noindent{\em Case~2-a. $v_2' - \mb e_2 + \mb e \notin \mathcal C_{\mb n_1 + \mb n_2}(S)$ for some $\mb e \in \{\pm \mb e_j: j \le d\} \setminus \{-\mb 
e_1, \pm \mb e_2\}$}. Notice that $v_2'  + \mb e \in N_1(v_2') \setminus \Pi'$. Since $N_1(v_2')$ does not intersect $\mathcal C_{\mb n_1 + \mb n_2}(S)$, it follows that the path $(v_2' - \mb e_2, v_2' - \mb e_2 + \mb e, v_2' + \mb e)$ lies in $N_2(\Pi') \setminus (\Pi' 
\cup \mathcal C_{\mb n_1 + \mb n_2}(S))$. Now $v_2' + \mb e$, being a neighbor of $v_2'$ other than $v_2' - \mb e_2$, has a neighbor in $N_1(v_3') \setminus \Pi'$ as we 
already saw above. Thus, like Case~1, we may choose $\Pi_{\mb B} = (v_0', v_1', \ldots, v_{t}')$ for $t = \min \{t' \le d_1: v_{t'}' \in \mathcal C_{\mb n_1}(y)\}$.

\medskip\noindent{\em Case~2-b. $v_2' - \mb e_2 + \mb e \in \mathcal C_{\mb n_1 + \mb n_2}(S)$ for each $\mb e \in \{\pm \mb e_j: j \le d\} \setminus \{-\mb 
	e_1, \pm \mb e_2\}$ and $d_1 > 3$}. Let us modify $\Pi'$ slightly to obtain a new path $\Pi'' \coloneqq (v_2' - 
\mb e_2 + \mb e_1, v_2' + \mb e_1, v_2' + \mb e_2 + \mb e_1, v_3', \ldots, v_{d_1}')$ and $\mathcal C_{\mb n_1 + \mb n_2}(S)$ restricted to the box of 
side-length $6dn$ with the same center as $\mb B$ and $u$. In the same spirit as in the case of $\Pi'$, the only problematic vertex in $N_2(\Pi'') \setminus (\Pi'' \cup \mathcal C_{\mb n_1 + \mb n_2}(S))$ is $(v_2' + \mb e_2 + 
\mb e_1) + \mb e_1 = v_3' + 2\mb e_1$. Hence we can apply the argument from Case~2-a to $\Pi''$ \emph{unless} $v_3' + 2\mb e_1 + \mb e \in \mathcal C_{\mb n_1 + \mb n_2}(S)$ for each $\mb e \in \{\pm \mb e_j: j \in [d]\} 
\setminus \{\pm \mb e_1, - \mb e_2\}$. 
But in that case, since $d_1 > 3$, there must be a vertex of $\mathcal C_{\mb n_1 + \mb n_2}(S)$ within a graph distance of 2 of $v_4$ and thus at a graph distance 
strictly smaller than $d_1$ of $u$ which contradicts the definition of $d_1$.

\medskip\noindent{\em Case~2-c. $v_2' - \mb e_2 + \mb e \in \mathcal C_{\mb n_1 + \mb n_2}(S)$ for each $\mb e \in \{\pm \mb e_j: j \le d\} \setminus \{-\mb 
	e_1, \pm \mb e_2\}$ and $d_1 = 3$}. In view of our discussion in the previous subcase, the only problematic scenario is the following. We have that $v_t' = v_0' + t \mb e_1$ for $t \leq 2$, $u = v_3' = v_2' + \mb e_2$ and $u + 2\mb e_1 + \mb e \in \mathcal C_{\mb n_1 + \mb n_2}(S)$ for each $\mb e \in \{\pm \mb e_j: j \in 
	[d]\} \setminus \{\pm \mb e_1, -\mb e_2\}$. Now recall that $u$ has at least two neighbors in $\mathcal C_{\mb n_1}(y)$ which, we claim, gives us in this case two vertices $w$ and $w'$ in $\mathcal C_{\mb n_1 + \mb n_2}(S)$ and $\mathcal C_{\mb n_1}(y)$ respectively such that $w' = w + 2 \mb e$ for some $\mb e \in \{\pm \mb e_j: j \le d\}$ and $w'$ is within a distance 1 of ${\bf B}$. Then it is straightforward to construct $\Pi_{\mb B}$ from $w$ and $w'$ as we already remarked before starting our case studies. In order to verify our claim, we need to consider three distinct possibilities for the two neighbors of $u$ that lie in $\mathcal C_{\mb n_1}(y)$.
	
	\medskip (i) $v_2' = u - \mb e_2 \in \mathcal C_{\mb n_1}(y)$. In this case we can choose $w = v_0'$ and $w' = v_2'$.
	
    \smallskip (iii) {\em $u + \mb e_1 \in \mathcal C_{\mb n_1}(y)$.} Here our choices are $w = v_2' - \mb e_2 + \mb e_1$ and $w' = u + \mb e_1$.
    
	\smallskip (iii) {\em $u + \mb e \in \mathcal C_{\mb n_1}(y)$ for some $\mb e \in \{\pm \mb e_j: j \leq d\} \setminus \{\pm \mb e_1, -\mb e_2\}$.} In this case we choose $w = u + 2\mb e_1 + \mb e$ and $w' = u + \mb e$.

\bigbreak
\noindent {\bf Case~3.} $d_\infty = 1$ and $d_1 \geq 3$.
\medbreak
In this case, let us assume without loss of generality that $v_1' = v_0' + \mb e_1$, $v_2' = v_1' + \mb e_2$ and 
$v_3' = v_2' + \mb e_3$. Notice that any vertex $p$ in $N_1(v_1') \setminus (\Pi' \cup \mathcal C_{\mb n_1 + \mb n_2}(S))$ other than $v_1' - \mb e_2$ has a neighbor in $N_1(v_2') \setminus \Pi'$, namely $p + \mb e_2$. Also noting that the vertices $v_1' - \mb e_2 + \mb e_3$ and $v_1' + \mb e_3$ lie in $N_2(v_3)\setminus \Pi'$ and $N_1(v_3)\setminus \Pi'$ respectively, we deduce that $(v_1' - \mb e_2, v_1' - \mb e_2 + \mb e_3, v_1' + \mb e_3)$ is a path in $N_2(\Pi') \setminus (\Pi' \cup \mathcal C_{\mb n_1 + \mb n_2}(S))$ if $v_1' - \mb e_2 \notin 
\mathcal C_{\mb n_1 + \mb n_2}(S)$. 

As to the vertices in $N_1(v_2') \setminus \Pi'$, we find by the same reasoning as in the analysis of Case~2-a that the only problematic scenario is $v_2' - \mb e_3 \notin \mathcal C_{\mb n_1 + \mb n_2}(S)$ and $v_2' - \mb e_3 + \mb e \in \mathcal C_{\mb n_1 + \mb n_2}(S)$ for every $\mb e \in \{\pm \mb e_j: j \le d\} \setminus \{-\mb e_2, \pm \mb e_3\}$. However, in this scenario we would have $\|(v_2' - \mb e_3 + \mb e_2) - u\|_1 = d_1$ whereas $\|(v_2' - \mb e_3 + \mb e_2) - u\|_\infty = 2$, thus reducing the problem to Case~2 with $v_0' = v_2' - \mb e_3 + \mb e_2$, which belongs to $\mathcal C_{\mb n_1 + \mb n_2}(S')$ and is within a distance of at most $3dn$ of the center of ${\bf B}$.
\bigbreak
\noindent {\bf Case~4.} $d_\infty = 1$ and $d_1 = 2$.
\medbreak
Let us assume without loss of generality that $u = v_0' + 
\mb e_1 + \mb e_2$ and $v_1' = v_0' + \mb e_1$. In this case we only need to connect any vertex in $N_1(v_1') \setminus (\Pi' \cup \mathcal C_{\mb n_1 + \mb n_2}(S))$ 
to a vertex in $N_1(u) \setminus \Pi'$. To this end notice that any vertex $p$ in $N_1(v_1') \setminus (\Pi' \cup \mathcal C_{\mb n_1 + \mb n_2}(S))$ other than $v_1' - \mb e_2$ has a neighbor $p + \mb e_2$ in $N_1(u) \setminus 
\Pi'$. Now if $v_1' - \mb e_2 \notin \mathcal C_{\mb n_1 + \mb n_2}(S)$ and $u + \mb e$ is a neighbor of $u$ in $\mathcal C_{\mb n_1}(y)$ which is not $u + \mb e_2$ (recall that there are at least two of them), then it is easy to see that either the path $(v_1' - \mb e_2, v_1' - \mb e_2 + \mb e, v_1' + \mb e)$ lies in $N_2(\Pi') \setminus (\Pi' \cup \mathcal C_{\mb n_1 + \mb n_2}(S))$ or there exists a vertex $w$ in $\mathcal C_{\mb n_1 + \mb n_2}(S)$ satisfying $w = (u + \mb e) + 2\mb e'$ or $w = (u + \mb e) + \mb e'$ some $\mb e' \in 
\{\pm \mb e_j: j \le d\}$. In both cases, the choice of $\Pi_B$ is clear.
\end{proof}
 

\begin{thebibliography}{ADCS15}

\bibitem[ABF87]{ABF87}
M.~Aizenman, D.~J. Barsky, and R.~Fern{\'a}ndez, \emph{The phase transition in
  a general class of {I}sing-type models is sharp}, J. Stat. Phys. \textbf{47}
  (1987), no.~3-4, 343--374.
  

\bibitem[ADCS15]{AizDumSid15}
M.~Aizenman, H.~Duminil-Copin, and V.~Sidoravicius, \emph{\mbox{Random} {C}urrents and
  {C}ontinuity of {I}sing {M}odel's {S}pontaneous \mbox{{M}agnetization}}, Comm. Math.
  Phys. \textbf{334} (2015), 719--742.

\bibitem[AF86]{AF86}
M.~Aizenman and R.~Fern{\'a}ndez, \emph{On the critical behavior of the
  magnetization in high-dimensional {I}sing models}, J. Stat. Phys. \textbf{44}
  (1986), no.~3-4, 393--454.

\bibitem[Aiz82]{Aiz82}
M.~Aizenman, \emph{Geometric analysis of $\varphi^4$ fields and {I}sing
  models.}, Comm. Math. Phys. \textbf{86} (1982), no.~1, 1--48.

\bibitem[ADTW18]{AizDumTasWar18}
M.~Aizenman, H.~Duminil-Copin, V.~Tassion, S.~Warzel, 
\emph{Emergent Planarity in two-dimensional Ising Models with finite-range Interactions}, arXiv:1801.04960, 2018.

\bibitem[Ale98]{Alexander98}
Kenneth~S. Alexander, \emph{On weak mixing in lattice models}, Probab. Theory
Related Fields \textbf{110} (1998), no.~4, 441--471. 

\bibitem[BD12]{BefDum12}
V.~Beffara and H.~{Duminil-Copin},
\newblock The self-dual point of the two-dimensional random-cluster model is
  critical for {$q\geq 1$}.
\newblock {\em Probab. Theory Related Fields}, 153(3-4):511--542, 2012.

\bibitem[Bod05]{Bod05}
T.~Bodineau, \emph{Slab percolation for the Ising model}, Prob. Theor. Rel. fields, {\bf 132(1)}, 83--118, 2005.
\bibitem[Bod06]{Bod06}
T.~Bodineau, \emph{Translation invariant {G}ibbs states for the {I}sing model},
  Prob. Theory Rel. Fields \textbf{135} (2006), no.~2, 153--168.

\bibitem[CCN87]{CCN87}
J.T. Chayes, L. Chayes and C.M. Newman, \emph{Bernoulli percolation above threshold: an invasion percolation analysis}, 
Ann. Probab. (1987) \textbf{15}, 1272--1287.

\bibitem[DP96]{deuschel1996surface}
J.~Deuschel and A.~Pisztora, \emph{Surface order large deviations for
  high-density percolation}, Probability Theory and Related Fields \textbf{104}
  (1996), no.~4, 467--482.
  
  \bibitem[Dum17]{Dum17}
H.~{Duminil-Copin}, \newblock {\em Lectures on the Ising and Potts models on the hypercubic lattice}, \newblock arXiv:1707.00520, 2017.

 \bibitem[DRT17]{DumRaoTas17}
H.~Duminil-Copin, A.~Raoufi, and V.~Tassion,
\newblock {\em Sharp phase transition for the random-cluster and potts models via
  decision trees}.
\newblock arXiv:1705.03104, 2017.

\bibitem[DT16]{DumTas15}
H.~{Duminil-Copin} and V.~Tassion, \emph{A new proof of the sharpness of the
  phase transition for {B}ernoulli percolation and the {I}sing model},
  Communications in {M}athematical {P}hysics \textbf{343} (2016), no.~2,
  725--745.

\bibitem[DT17]{DumTas17}
H.~{Duminil-Copin} and V.~Tassion, \emph{A note on Schramm's locality conjecture for random-cluster models}, arXiv:1707.07626, 2017.
    
\bibitem[DS87]{DobShl}
R.L. Dobrushin and S.B. Shlosman, \emph{Completely analytical interactions:
  Constructive description}, J. Stat. Phys. \textbf{46} (1987), no.~5/6, 983 --
  1014.

\bibitem[FK72]{ForKas72}
C.~M. Fortuin and P.~W. Kasteleyn, \emph{On the random-cluster model. {I}.
  {I}ntroduction and relation to other models}, Physica \textbf{57} (1972),
  536--564.
    
\bibitem[FKG71]{FKG71} 
C. M. Fortuin, P. W. Kasteleyn, and J. Ginibre, 
\emph{Correlation inequalities on some partially ordered sets},  
Comm. Math. Phys. \textbf{22} (1971), 89.


\bibitem[GHS70]{GHS}
R.B. Griffiths, C.A. Hurst, and S.~Sherman, \emph{Concavity of magnetization in
  {I}sing ferromagnets in a positive external field}, J. Math. Phys.
  \textbf{11} (1970), 790.

\bibitem[Gr67a]{Gri67}
R.B. Griffiths, \emph{Correlations in {I}sing ferromagnets. {I}}, 
J. Math. Phys. \textbf{8} (1967), 478.

\bibitem[Gr67b]{Grif_ghost}
R.B. Griffiths, \emph{Correlations in {I}sing ferromagnets. {II}. external
  magnetic fields}, J. Math. Phys. \textbf{8} (1967), 484.

\bibitem[Gri06]{Gri06}
G.~Grimmett, \emph{The random-cluster model}, Grundlehren der Mathematischen
  Wissenschaften [Fundamental Principles of Mathematical Sciences], vol. 333,
  Springer-Verlag, Berlin, 2006.


\bibitem[GM90]{GM90}
G.R. Grimmett and J.M. Marstrand, {\em The supercritical phase of percolation is well behaved}, Proc. R. Soc. Lond. A (1990), {\bf 430}, 439--457.


\bibitem[LML72] {LML72} 
J.L. Lebowitz and A. Martin-L\"of, \emph{ On the Uniqueness of the Equilibrium State
for Ising Spin Systems}, Comm. Math. Phys. \textbf{25} (1972), 
276--282.  

\bibitem[LP68]{LP68}
J.L. Lebowitz and O.~Penrose, \emph{Analytic and clustering properties of
  thermodynamic functions and distribution functions for classical lattice and
  continuum systems}, Comm. Math. Phys. \textbf{11} (1968), 99--124.

\bibitem[LS13]{LubSly13}
E.~Lubetzky and A.~Sly, \emph{Cutoff for the {I}sing model on the lattice},
  Invent. Math. \textbf{191} (2013), 719 -- 755.

\bibitem[LSS97]{LigSchSta97}
T.~M. Liggett, R.~H. Schonmann, and A.~M. Stacey, \emph{Domination by product
  measures}, Ann. Probab. \textbf{25} (1997), no.~1, 71--95. 

\bibitem[LY52]{LY52}
T.D. Lee and C.N. Yang, \emph{Statistical theory of equations of state and
  phase transitions. {II.} lattice gas and {I}sing model}, Phys. Rev.
  \textbf{87} (1952), 410--419.


\bibitem[MOS94]{MartinelliOlivSchonmann}
F.~Martinelli, E.~Olivieri, and R.~H. Schonmann, \emph{For {$2$}-{D} lattice spin systems weak mixing implies strong mixing}, Comm. Math. Phys. \textbf{165} (1994), 33--47.

\bibitem[MW73]{McCWu73}
B.M. McCoy and T.T. Wu, \emph{The two-dimensional {I}sing model}, Harvard
  University Press, Cambridge, MA, 1973.

\bibitem[Pis96]{Pis96}
Agoston Pisztora, \emph{Surface order large deviations for {I}sing, {P}otts and
  percolation models}, Probab. Theory Related Fields \textbf{104} (1996),
  no.~4, 427--466. 
  
  \bibitem[Rao17]{Rao17}
   A. Raoufi, \emph{Translation-invariant Gibbs states of Ising model: general setting}, arXiv:1710.07608, 2017.
   
\bibitem[Rei00]{BKR_ineq}
 D. Reimer, \emph{Proof of the Van den Berg-Kesten conjecture}, Combin. Probab. Comput. \textbf{9} (2000), 27--32.


\bibitem[Sim80]{Sim80}
Barry Simon, \emph{Correlation inequalities and the decay of correlations in
  ferromagnets}, Comm. Math. Phys. \textbf{77} (1980), no.~2, 111--126.

\bibitem[vdBK85] {vdBK85}
 J. van den Berg and H. Kesten, \emph{Inequalities with applications to percolation and
reliability}, J. Appl. Probab. \textbf{22} (1985), 556--569.

\end{thebibliography}
\end{document}